\documentclass[12pt]{amsart}
\usepackage{amscd,amsthm,amssymb,amsfonts,amsmath,epsfig,url, latexsym, tikz}
\usetikzlibrary{arrows}

\theoremstyle{plain}
\newtheorem{theorem}{Theorem}
\newtheorem{lemma}[theorem]{Lemma}
\newtheorem{proposition}[theorem]{Proposition}
\newtheorem{corollary}[theorem]{Corollary}
\newtheorem{conjecture}[theorem]{Conjecture}

\theoremstyle{definition}
\newtheorem{definition}[theorem]{Definition}
\newtheorem{example}[theorem]{Example}

\theoremstyle{remark}
\newtheorem*{remark}{Remark}
\newtheorem*{acknowledgments}{Acknowledgments}

\newcommand{\Z}{\mathbb Z}    
\newcommand{\R}{\mathbb R}    
\newcommand{\C}{\mathbb C}    
\newcommand{\PP}{\mathbb P}   
\newcommand{\T}{\mathbb T}    

\renewcommand{\O}{\mathcal O}   

\newcommand{\suchthat}{\ : \ }
\newcommand{\<}{\langle}   
\renewcommand{\>}{\rangle} 

\newcommand{\Log}{\operatorname{Log}}

\newcommand{\Arg}{\operatorname{Arg}}
\newcommand{\am}{{\mathcal{A}}}
\newcommand{\ignore}[1]{\relax}

\newcommand{\rk}{\operatorname{rk}}

\newcommand{\TT}{\mathbb{T}}

\newcommand{\chull}{\operatorname{Conv}}

\newcommand{\RR}{\mathcal{R}}
\newcommand{\pp}{\mathcal{P}}
\newcommand{\cc}{\mathcal{C}}
\newcommand{\hh}{\mathcal{H}}
\renewcommand{\tt}{\mathcal{T}}

\newcommand{\tropvar}[1]{\mathcal{H}_{#1}}
\newcommand{\tropstrat}[2]{\mathcal{H}_{{#1}, {#2}}}
\newcommand{\tropstrats}[3]{\mathcal{T} \mathcal{H}_{{#1}, {#2}, {#3}}}
\newcommand{\ctropstrats}[3]{\mathcal{T} \bar{\mathcal{H}}_{{#1}, {#2}, {#3}}}
\newcommand{\comhyp}[1]{H_{#1}}
\newcommand{\ccomhyp}[1]{\bar{H}_{#1}}
\newcommand{\comstrat}[2]{H_{{#1}, {#2}}}
\newcommand{\ccomstrat}[2]{\bar{H}_{{#1}, {#2}}}
\newcommand{\coam}[1]{\mathcal{C}_{#1}}
\newcommand{\tphyp}[2]{\mathcal{T}\mathcal{H}_{#1, #2}}
\newcommand{\ctphyp}[2]{\mathcal{T}\bar{\mathcal{H}}_{#1,#2}}

\begin{document}

\title{Phase tropical hypersurfaces}

\author{Gabriel Kerr and Ilia Zharkov}
\address{Kansas State University, 138 Cardwell Hall, Manhattan, KS 66506}
\email{gdkerr@ksu.edu, zharkov@ksu.edu}

\begin{abstract}
We prove a conjecture of Viro \cite{Viro} that a smooth complex hypersurface in $(\C^*)^n$ is homeomorphic to the corresponding phase tropical hypersurface.
\end{abstract}
\maketitle
\section{Introduction}
Consider a hypersurface $H_f\subset (\C^*)^n$ defined by a Laurent polynomial 
$$f = \sum_{a \in A} c_a z^a,$$
where $A\subset \Z^n$ is the set of monomials. Let $Q$ be its Newton polytope, that is, the convex hull of $A$. Given a function $\eta: A \to \R$, whose upper graph induces a triangulation of $Q$, one considers the associated phase tropical hypersurface $\tt\hh_\eta \subset \R^n \times \TT^n \cong (\C^*)^n$. This is a polyhedral object which surjects onto the tropical hypersurface $\hh_\eta \subset \R^n$. 
Over the relative interior of a face of $\hh_\eta$ dual to a simplex $Q'$ in the triangulation, the fiber of this surjection is the coamoeba of the truncated hypersurface 
 $$\sum_{a \in \operatorname{vert} Q'} c_az^a=0.
 $$
Our main result (cf. Theorem \ref{thm:maintheorem}) states that for a generic polynomial $f$, the complex  hypersurface $H_f$ is homeomorphic to the phase tropical hypersurface $\tt\hh_\eta$, which was a conjecture of Viro \cite{Viro}. 

In Section \ref{sec:H} we reduce the case of a general hypersurface to finite abelian coverings of a pair-of-pants using Viro's patchworking \cite{Viro83} and a non-unimodular version of the Mikhalkin's pair-of-pants decomposition \cite{PP}. 
Thus, the core of the proof is the case of pair-of-pants.
The closures $\bar P^{n-1}$ and $\bar \tt\pp^{n-1}$ of the corresponding pairs-of-pants carry natural stratifications from the ambient space $\Delta\times \TT^n$. The key technical result of Section~\ref{sec:PP} is that the closed strata are balls.
A homeomorphism $\bar P^{n-1} \approx \bar \tt\pp^{n-1}$ then follows from an isomorphism between the two regular CW-complexes, one for~$\bar P^{n-1}$ and the other for~$\bar \tt\pp^{n-1}$. 

In the final stages of writing the paper we were made aware of an announcement by Kim and Nisse of similar results in Theorem 1.1 and Proposition 5.2 of \cite{KN16}. 

\begin{acknowledgments}
We would like to thank Jianting Huang, Grisha Mikhalkin, Mounir Nisse, David Nadler, Nick Sheridan, Oleg Viro and Peng Zhou for very fruitful conversations on the subject of the paper. Christian Haase shared some essential ideas on how to prove that a polyhedral complex is homeomorphic to the ball. We would also like to thank Dave Auckly, Igor Belegradek and Andreas Thom for helping us with the cobordism argument in the proof of Proposition \ref{prop:the_ball}. Finally we would like to thank the anonymous referees for pointing out some mistakes and suggesting several improvements of the original manuscript. G.K. acknowledges support by the Simons collaboration grant. The research of I.Z. is partially supported by the NSF FRG grant DMS-1265228. 
\end{acknowledgments}

\section{Pair-of-pants}\label{sec:PP}
The main result of this section is a homeomorphism between the complex pair-of-pants and the phase tropical pair-of-pants, cf. Theorem  \ref{thm:homeo}. The idea is to endow both spaces with structures of regular CW-complexes which are isomorphic.

\subsection{Notations} 
Throughout the paper we identify $\C^*$ with $\R\times (\R/2\pi \Z)$. In particular, we will identify $(\C^*)^{n+1}/\C^*$ with 
$$\left( \R^{n+1}/\R \right) \times  \left( (\R/2\pi \Z)^{n+1}  /(\R/2\pi \Z ) \right),
$$ 
where both $\R$ and $\R/2\pi \Z$ act diagonally.
We denote the second factor by
$$\TT^n := (\R/2\pi \Z)^{n+1}/(\R/2\pi \Z) \cong \R^n/2\pi\Z^n.
$$ 

We will use homogeneous (additive for the last two cases) coordinates 
\begin{equation}
\begin{gathered}
\label{eq:coordinates}
[z_0, \dots, z_n] \text{ in } (\C^*)^{n+1}/\C^*,\\
[x_0, \dots, x_n] \text{ in } \R^{n+1}/\R,\\
[\theta_0, \dots, \theta_n] \text{ in } \TT^n.
\end{gathered}
\end{equation}
An element in $\T^n$ can be thought of as a configuration of marked points $\theta_{0}, \theta_{1}, \dots, \theta_{n}$ on the unit circle up to simultaneous rotation.

Let $\hat n$ denote the set $\{0,\dots,n\}$. For any subset $I\subseteq \hat n$ we denote by $I^c$ its complement. We denote  by $\pi_I=[\theta_0, \dots, \theta_n]$ the point in $\TT^n$ with coordinates
\begin{equation}\label{eq:pi}
\theta_i= \
\begin{cases} 
\pi,\  i \in I,\\ 
0, \ i \not \in I.
\end{cases}
\end{equation}
The points $\pi_I$ and $\pi_{I^c}$ coincide. 
The origin $[0, \dots, 0]$ is denoted by $0$.

Let 
$$\Delta:=\left\{(y_0, \dots, y_n) \in \R^{n+1} \suchthat y_i\ge 0, \sum y_i =1 \right\}
$$ 
be the standard $n$-simplex. For a non-empty subset $J\subseteq \hat n$ the face $\Delta_J$ of $\Delta$ is defined by $y_i=0, i \in J^c$. We will identify $\R^{n+1}/\R$ with the interior of $\Delta$ via the map
\begin{equation}\label{eq:moment}
[x_0,  \dots , x_n] \mapsto 
\left(
\frac{e^{x_0}}{e^{x_0}+\dots+e^{x_n}},\dots, \frac{e^{x_n}}{e^{x_0}+\dots+e^{x_n}}
\right).
\end{equation}
Multiplying by the factor $\TT^n$ leads to a compactification of $(\C^*)^{n+1}/\C^*$ to the 
space $\Delta \times \TT^n$. For any subset $Y\subset (\C^*)^{n+1}/\C^*$ we define its {\bf compactified} version $\bar Y$ to be the closure of $Y$ in $\Delta \times \TT^n$ via the map \eqref{eq:moment} above.

\subsection{The face lattice $\mathcal W$ of the future CW-complex} 
We say that 
$\sigma =\< I_1, \dots, I_k\>$
 is a {\bf cyclic partition} of the set $\hat n= \{0,\dots,n\}$ if  $\hat n$ is a disjoint union of the sets $I_1, \dots, I_k$ and the sets $I_1, \dots, I_k$ are cyclically ordered. The elements within each $I_s$ are not ordered. If all $I_s$ are 1-element sets then we simply write $\sigma=\<i_0, \dots, i_n\>$.
Our main source of cyclic partitions of $\hat n$ will be configurations of marked points $\theta_{0}, \theta_{1}, \dots, \theta_{n}$ on the oriented circle. 

The set of hyperplanes $\theta_i=\theta_j,\ i,j\in \hat n$, stratifies the torus $\T^n$ with strata $\T^n_\sigma$ labeled by cyclic partitions $\sigma$. 
On the other hand the simplex $\Delta$ has a natural stratification by its faces $\Delta_J$. The product of the two stratifications induces a stratification on any closed subset $\bar Y \subseteq  \Delta \times \T^n$. The strata $Y_{\sigma,J}$ of $\bar Y$ are labeled by the pairs $(\sigma, J)$, where $\sigma =\<I_1, \dots, I_k\>$ is a cyclic partition of $\hat n$ and $J \subseteq \hat n$. 
The inclusion of the strata closures $\bar Y_{\sigma',J'}\subseteq \bar Y_{\sigma,J}$ gives a partial order among the pairs: $(\sigma', J') \preceq (\sigma, J)$ if $\sigma$ is a refinement of $\sigma'$ (we write $\sigma'\preceq \sigma$) and $J'\subseteq J$.

To simplify notations we will often drop the index $J$ from the subscript if $J=\hat n$.
For any non-empty subset $J\subseteq \hat n$ a cyclic partition $\sigma=\< I_1, \dots, I_k\>$ of $\hat n$ induces a cyclic partition $\sigma_J=\< J_1, \dots, J_r\>$ of $J$ by intersecting each $I_s$ with $J$. We will drop the empty intersections and shift the indices, in this case $r$ will be smaller than $k$. 

Our main focus will be on the poset $\mathcal W$ which consists of pairs $(\sigma, J)$ such that $J$ contains elements in at least two of the subsets $I_1, \dots, I_k$ of $\sigma=\<I_1, \dots, I_k\>$. 
In this case we say that $\sigma$ {\bf divides} $J$ and write $\sigma | J$. This, in particular, means that $k\ge 2$ and $|J|\ge 2$.  We set the rank function to be 
$$\rk (\sigma,J) := k + |J| -4.
$$ 
The poset $\mathcal W$ will be the face lattice of our regular CW-complex and $\rk (\sigma,J)$ will be the dimension of the $(\sigma,J)$-cell.

\begin{conjecture}\label{conj:grunbaum}
For each element $(\sigma,J)\in \mathcal W$ its lower interval $\mathcal W_{\preceq(\sigma,J)} :=\{(\sigma',J')\in \mathcal W \suchthat (\sigma',J') \preceq (\sigma, J)\}$ is isomorphic to the face lattice of a simple polytope.
\end{conjecture}

\noindent
It is clear that for any pair $(\sigma', J') \preceq (\sigma, J)$ the interval $[(\sigma', J'), (\sigma, J)]$ is Boolean, which means that the polytope would have to be simple. 
The conjecture is manifest for $n=2$: maximal faces $\mathcal W_{\preceq \sigma}$ are hexagons. For $n=3$ each maximal face $\mathcal W_{\preceq \sigma}$ is the 4-dimensional polytope with 20 vertices and 8 facets, dual to $P^8_{35}$, one of the 37 simplicial polytopes on 8 vertices classified by Gr\"unbaum and Sreedharan \cite{GS67}. The next problem would be to realize $\mathcal W_{\preceq \sigma}$ inside a linear space, which is already interesting for $n=2$ and 3.

\subsection{Complex pair-of-pants as a CW complex} 
The $(n-1)$-dimensional {\bf pair-of-pants} $P^{n-1}$ is the complement of $n+1$ generic hyperplanes in $\C\PP^{n-1}$. 
By an appropriate choice of coordinates we can identify $P^{n-1}$ with the affine hypersurface in $(\C^*)^{n+1}/\C^*$ given by the homogenous equation 
$$z_0+z_1+\dots + z_n=0.
$$ 
We define the {\bf compactified pair-of-pants} $\bar P^{n-1}$ to be the closure of $P^{n-1}$ in $\Delta \times \TT^n$ via the map \eqref{eq:moment}. This is a manifold with corners, and it can be thought of as a real oriented blow-up of $\C\PP^{n-1}$ along its intersections with the coordinate hyperplanes in $\C\PP^n$.

We can view points in $P^{n-1}$ as closed oriented broken lines with $n+1$ marked segments in the plane defined up to rigid motions and scaling. The segments represent the complex numbers $z_0, \dots, z_n$. In the compactification $\bar P^{n-1}$ the broken lines may have sides of zero length but with directions still recorded.

Recall that the $(\sigma,J)$-stratification on $\Delta \times \T^n$ induces a stratification on any closed subset of $\Delta \times \T^n$, in particular on $\bar P^{n-1}$. We denote by $\Phi_{\sigma, J}$ the corresponding stratum of $\bar P^{n-1}$ and by $\bar \Phi_{\sigma, J}$ its closure in $\bar P^{n-1}$.

For $[z_0,\dots, z_n]$, a point in a stratum $\Phi_{\sigma, J}$ of $\bar P^{n-1}$, one can rearrange the variables such that their arguments are (partially) ordered counter-clockwise on the circle. The order of the $z_i$ is defined up to permutations within the subsets $I_s$ in $\sigma=\<I_1, \dots, I_k\>$.
Then the circuit of vectors $z_{i_0}, \dots, z_{i_n}$ forms a convex (possibly degenerate) polygon $\mathcal D$ in the plane (see Fig.~\ref{fig:polygon}). The vertices of $\mathcal D$ separate the subsets $I_s$ in~$\sigma$. 
\begin{figure}[h]
\centering
\includegraphics[height=35mm]{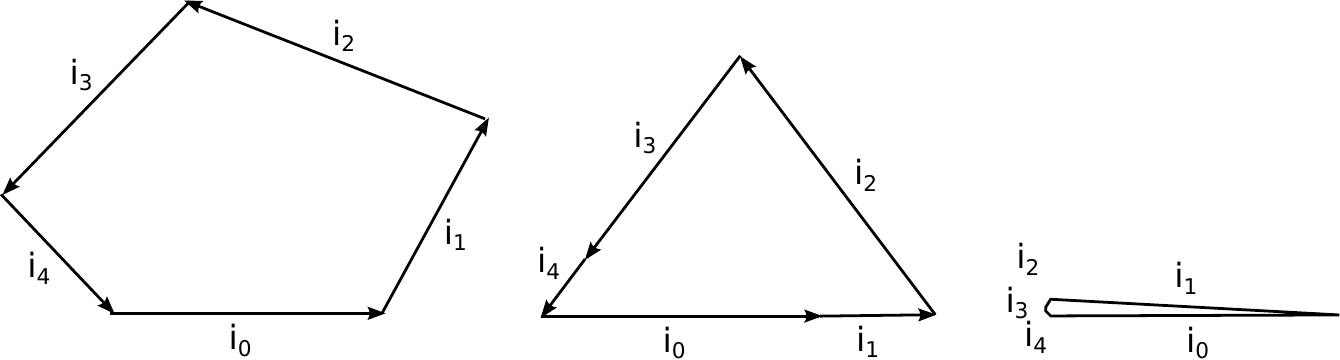}
\caption{Polygons represent points in $\Phi_{\<i_0, \dots, i_4\>}$, $\Phi_{\<\{i_0,  i_1\}, i_2, \{i_3, i_4\}\>}$ and $\Phi_{\<i_0, \dots, i_4\>, J=\{i_0,i_1\}}$.}
\label{fig:polygon}
\end{figure}

One can deform a polygon $\mathcal D$ representing a point in $\Phi_{\sigma, J}$ by ``bending'' its edges within each $I_s$ (that refines $\sigma$) and introducing small lengths for zero edges (that increases $J$). Thus we have the following observation.

\begin{proposition}\label{prop:boundary}
A closed stratum $\bar \Phi_{\sigma, J}$ contains  $\Phi_{\sigma', J'}$ if and only if $(\sigma', J') \preceq (\sigma, J)$. 
\end{proposition}

Next we argue that the $(\sigma,J)$-stratification defines a CW structure on $\bar P^{n-1}$.

\begin{lemma}\label{lemma:interior_ball}
$\Phi_{\sigma, J}$ is homeomorphic to $\R^{\rk(\sigma,J)}$ if $(\sigma,J) \in \mathcal W$, and it is empty if $(\sigma,J) \not\in \mathcal W$.
\end{lemma}

\begin{proof}
If $(\sigma,J) \not\in \mathcal W$ then the set of non-zero edges $J$ falls in a single subset $I_s$ of $\sigma=\<I_1,\dots, I_k\>$. But it is impossible to build a closed circuit with just one non-zero side.

Now let $(\sigma,J) \in \mathcal W$ be a maximal strata, that is $\sigma=\<i_0, \dots, i_n\>$ and $J=\hat n$ (remember we drop the subscript $J$ from $\Phi_{\sigma, J}$ in this case).
We set $z_{i_0}=1$. That fixes the rotational and scaling ambiguity and we can think of $\Phi_{\sigma}$ as a subset in $(\C^*)^n$.

Denote by $\Phi_{\sigma}^{(r)}\subset (\C^*)^r$ the image of $\Phi_{\sigma}$ under the projection onto the first $r$ coordinates $z_{i_1},\dots, z_{i_r}$. 
Notice that $\Phi_{\sigma}^{(1)}$ is the upper-half plane. 
For $0< r<n-1$ the fiber of the projection $\Phi_{\sigma}^{(r+1)} \to \Phi_{\sigma}^{(r)}$ over a point $(z_{i_1},\dots, z_{i_r})$ 
is an open polyhedral domain in the plane defined by 3 linear inequalities (the red region in Fig. \ref{fig:fiber}). \begin{figure}[h]
\centering
\includegraphics[height=35mm]{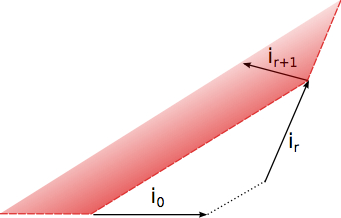}
\caption{Linear inequalities for $z_{i_{r+1}}$ defining the fiber.}
\label{fig:fiber}
\end{figure}
Finally for $r=n-1$ the fiber is a point: the last vector $z_{i_n}$ has to close the circuit. By induction this shows that the $\Phi_{\sigma}$ is homeomorphic to $\R^{2n-2}$.

For general $\sigma$ and $J$ one can first replace the vectors in each part $J_s$ of the induced cyclic partition $\sigma_J=\<J_1, \dots, J_l\>$ by their sum, thus reducing the number of edges to $l$. This projects $\Phi_{\sigma,J}$ to a lower dimensional maximal case, which is $\R^{2l-4}$ by the previous argument.
A fiber of this projection consists of possible splittings of the edge vectors into several parallel non-zero vectors from the same $J_s$, which gives $\R^{|J|-l}$, plus choosing the arguments $\theta_{I_s}$ for 
the subsets $I_s$ with $I_s\cap J=\emptyset$, according to their order in $\sigma$, which gives another $\R^{k-l}$. Putting it all together we conclude that the total space is $\R^{k+|J|-4}$.
\end{proof}

\begin{lemma}\label{lemma:manifold}
Each closed stratum $\bar\Phi_{\sigma, J}$ is a topological manifold with boundary. 
\end{lemma}

\begin{proof}
Let $\mathcal D$ be a $k$-gon which represents some point in a stratum $\Phi_{\sigma',J'}$ in $\bar{\Phi}_{\sigma,J}$. Here $\sigma'=\<I'_1,\dots, I'_{k'}\>$ is a coarsening of  $\sigma=\<I_1,\dots, I_k\>$ and $J'\subseteq J$.
We describe a coordinate system in a neighborhood of $\mathcal D$ in $\bar{\Phi}_{\sigma,J}$ which maps it to a neighborhood of a corner point in 
\begin{equation}\label{eq:coords}
\R_{\ge 0}^{|J|-|J'|} \times \R^{|J'|-2} \times \R_{\ge 0}^{k-k'} \times \R ^{k'-2}.
\end{equation}

We choose $\<I_-,I_+\>$, a cyclic 2-partition coarsening of $\sigma'$, and two elements $j_\pm \in J'_\pm:= I_\pm\cap J' $ (remember, $\sigma'$ divides $J'$). Set $J_\pm := I_\pm \cap J$. Let $\mathcal V_\pm$ be the sets vertices of $\mathcal D$ which separate subsets of $\sigma'$ in $I_\pm$, respectively. Together there are $k'-2$ vertices in $\mathcal V_-$ and $\mathcal V_+$. Let $\mathcal V'$ be the set of vertices of $\mathcal D$ which separate subsets of $\sigma$ inside the subsets of $\sigma'$. There are $k-k'$ vertices in $\mathcal V'$.

The first $|J|-2$ coordinates are given by the lengths of edges in $J_-, J_+$ relative the lengths of ${j_-}, {j_+}$, respectively. Namely we set 
$x_j := |z_j|/|z_{j_\pm}|, \ j\in J_\pm\setminus j_\pm.
$
Note that $x_j=0, j\in J\setminus J'$, at $\mathcal D$ and they can deform only positively. The coordinates $x_j$ give the first two factors in \eqref{eq:coords}. The last two factors in \eqref{eq:coords} are formed by the exterior angles $\alpha_r$ at the vertices $\mathcal V_\pm$ and $\mathcal V'$ of $\mathcal D$. The angles at $\mathcal V'$ are zero at $\mathcal D$ and can only deform positively to maintain convexity of nearby polygons in $\bar{\Phi}_{\sigma,J}$.

\begin{figure}[h]
\centering
\includegraphics[height=30mm]{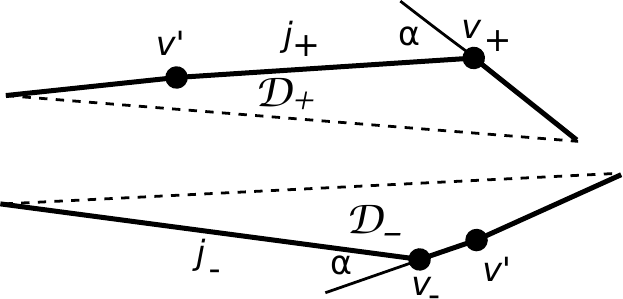}
\caption{Gluing polygon from its two deformed halves.}
\label{fig:coordinates}
\end{figure}

Any small variation of $x_i$'s and $\alpha_r$'s from the original values at $\mathcal D$ will independently deform the two halves $\mathcal D_\pm$, which correspond to $I_\pm$ (see Fig. \ref{fig:coordinates}). 
Then one uniquely reconstructs a polygon $\mathcal D'$ by rescaling (the values of the $x_i$'s and $\alpha_r$'s are not changed) and gluing the deformed halves~$\mathcal D'_\pm$ at the ends. 
\end{proof}

\begin{remark}
The above argument shows that $\bar\Phi_{\sigma, J}$ is, in fact, a manifold with corners.
\end{remark}

\begin{proposition}\label{prop:the_ball}
Each closed stratum $\bar \Phi_{\sigma,J}$ is homeomorphic to a closed ball.
\end{proposition}

\begin{proof}
Lemmas \ref{lemma:interior_ball} and \ref{lemma:manifold} show that $\bar \Phi_{\sigma,J}$ is a compact topological manifold with boundary whose interior is homeomorphic to the Euclidean space.  
In particular its boundary $\partial \bar \Phi_{\sigma,J}$ is simply connected (unless it is of dimension one). We can remove a topological ball from the interior and use a collar of the boundary to get an $h$-cobordism between the boundary and the standard sphere which has to be trivial, at least in dimension $> 4$. Gluing the ball back in we conclude that $\bar \Phi_{\sigma,J}$ has to be a closed ball. In dimensions $\le 4$ one can give an explicit homeomorphism of  $\bar \Phi_{\sigma,J}$ with a simple polytope (cf. remark after Conjecture~\ref{conj:grunbaum}).
\end{proof}

Combining Propositions \ref{prop:boundary} and \ref{prop:the_ball} we arrive at the desired CW decomposition of the pair-of-pants (cf., e.g., \cite{LW} for details about regular CW complexes). 

\begin{proposition}\label{prop:complexCW}
 $\bar P^{n-1}= \bigcup_{(\sigma, J) \in \mathcal W} \Phi_{\sigma, J}$ is a regular CW-complex.
\end{proposition}

\subsection{The coamoeba and its decompositions}
\label{sec:coamoeba}
Consider the argument map 
$$\Arg: (\C^*)^{n+1}/\C^* \to \TT^n, \qquad  [z_0 , \dots, z_n] \mapsto [\arg(z_0), \dots , \arg(z_n)].
$$
The closure of the image $\Arg (P^{n-1})$ in $\TT^n$ is the {\bf coamoeba} $\cc^n$ of the pair-of-pants. 
The argument map extends to a continuous surjective map $\Arg: \bar P^{n-1} \to \cc^n$ via the projection from $\Delta \times \TT^n$ onto the second factor. 

We can think of points in $\cc^n$ as allowed configurations of $n+1$ marked points $\theta_0, \dots, \theta_n$ on the circle. A configuration is {\bf allowed} if not all points lie on an open half-circle. Any allowed configuration is realized by a point in $\bar P^{n-1}$: we circumscribe a polygon $\mathcal D$ around the circle with edges tangent at the $\theta_i$'s.
Excluding {\em non}-allowed configurations leads to a well-known  description of $\cc^n$ as the complement of the interior of the zonotope  
(cf., e.g., \cite[Prop 2.1]{Sheridan})
$$Z=  \sum_{i=0}^n [0, \pi_{i}],
$$ 
where an interval $[0, \pi_{i}] \subset \TT^n$ is defined by $\theta_0{=} \dots \hat\theta_i \dots  {=} \theta_n$, $\theta_i-\theta_0 \in [0,\pi]$. The facets of $Z$ are given by hyperplanes $\theta_i-\theta_j=\pi$.
Among all boundary points of $\cc^n$ only the vertices $\pi_I$, for $I\ne \emptyset$ or $\hat n$, are in the image $\Arg (P^{n-1})$.

For any subset  $J\subseteq \hat n$ we define the partial coamoeba $\cc_J$ to be the closure of $\Arg (P^{n-1}_J)$ in $\TT^n$, where $P^{n-1}_J \subset  (\C^*)^{n+1}/\C^*$ is the hypersurface given by 
$$\sum_{j\in J} z_j =0.
$$  
An allowed configuration of points on the circle remains allowed if we add more points to it. This shows that $\cc_I\subseteq \cc_J$ for $I\subseteq J$. 
In particular, all $\cc_J$ are closed subsets of $\cc^n$. Note that $\cc_J$ is empty unless $|J|\ge 2$.

We call a subset in $\TT^n$ a {\bf polytope} if it is a bijective image of a convex polytope in the universal cover $\R^{n+1}/\R$. We will often define a polytope by a set of inequalities in $\R^{n+1}/\R$ which depends on a cyclic partition $\sigma$ along with a choice of an initial subset in~$\sigma$. However the image polytope in $\cc^n$ will be independent of that choice. 
We give two polytopal  decompositions of $\cc^n$. The second is a refinement of the first. 

The {\bf octahedral} decomposition (the name comes from the case $n=3$, see Fig. \ref {fig:coamoeba}) is the restriction to $\cc^n$ of the stratification $\T^n_\sigma$ of $\T^n$ by the cyclic partitions of $\hat n$. For $\sigma= \<I_1, \dots, I_k\>$ the octahedron $\O_{\sigma}:= \cc^n \cap \bar \T^n_\sigma$ is given  (in $\R^{n+1}/\R$) by 
\begin{equation}\label{eq:octa}
\begin{split}
\theta_i=\theta_{i'} =: \theta_{I_s}, & \text{ for } i,i' \in I_s,\\
\theta_{I_s} \le \theta_{I_{s+1}} \le \theta_{I_s} + \pi, & \  s=1, \dots, k-1,
\\
\theta_{I_k} \le \theta_{I_{1}} + 2\pi  \le \theta_{I_k} + \pi &. 
\end{split}
\end{equation}
Left inequalities reflect the order of $\sigma$. Right inequalities define the boundary of $\cc^n$: they exclude non-allowed configurations.
Changing the initial subset from $I_1$ to $I_r$ in $\sigma$ would amount to shifting $\theta_1, \dots, \theta_{r-1}$ by $2\pi$.
Note that two distinct lifts to $\R^{n+1}/\R$ of a point in $\T^n$ cannot both satisfy \eqref{eq:octa}.
In particular, it means that $\O_{\sigma}$ is a polytope in $\TT^n$. 

The full-dimensional octahedra correspond to maximal cyclic partitions. In general, the dimension of $\O_\sigma$ is $k-1$, where $k$ is the number of sets in $\sigma$. The vertices are exceptions from this rule, they correspond to cyclic 2-partitions. And there are no 1-dimensional octahedra. 

\begin{remark}
The octahedral decomposition is not a polyhedral complex. The faces of octahedra on the boundary of $\cc^n$, except vertices, are not octahedra themselves. In particular, the edges always lie on the boundary. 
\end{remark}

\begin{example}
For $n=2$ there are 2 maximal octahedra (= triangles in this case). For $n=3$ there are 6 maximal octahedra. In Fig. \ref {fig:coamoeba} (we set $\theta_0=0$ and $0\le \theta_{i} \le 2\pi$) the red octahedron is $\O_{0213}$, the blue one is $\O_{0321}$ (we dropped commas and brackets from the subscripts). The triangle face common to the red and green octahedra is $\O_{02\{13\}}$. For $n\ge3 $ a maximal octahedron has $n+1$ pairs of facets (corresponding to $n+1$ pairs of inequalities \eqref{eq:octa}). The $n+1$ facets in the interior of $\cc^n$ are the $(n-1)$-maximal octahedra. Opposite to each such octahedron lies an $(n-1)$-simplex, which is on the boundary of $\cc^n$. 
\end{example}

\begin{figure}[h]
\centering
\includegraphics[height=50mm]{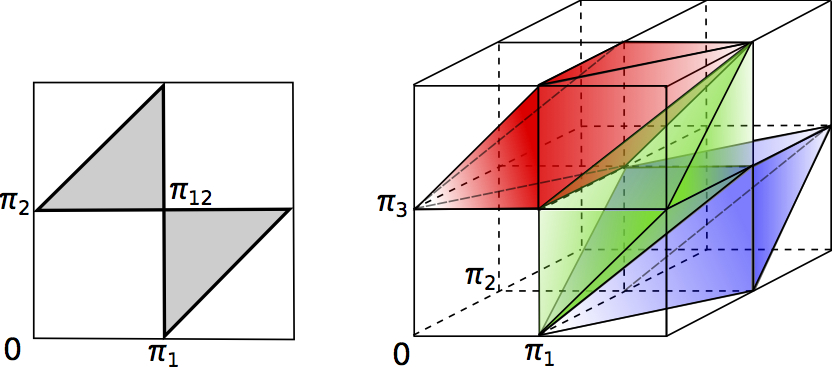}
\caption{Octahedral subdivisions: the two triangles for $n=2$, and three of the six octahedra for $n=3$.}
\label{fig:coamoeba}
\end{figure}

Let us look at the induced decompositions of partial coamoebas $\cc_J$. Let $\sigma=\<I_1, \dots, I_k\>$ and $\sigma_J=\<J_1, \dots, J_r\>$.
The intersection $\O_{\sigma,J}:= \cc_J\cap \O_\sigma$ is not an octahedron anymore, though it is still a polytope in $\TT^n$. Namely, it is cut out by the inequalities (in $\R^{n+1}/\R$): 
\begin{equation}\label{eq:octa_partial}
\begin{split}
\theta_i=\theta_{i'}  =:\theta_{I_s}, &  \text{ for } i,i' \in I_s,\\
\theta_{I_1}\le \dots \le \theta_{I_k}  \le \theta_{I_1}+2\pi, & \\ 
\theta_{j_{s+1}} \le \theta_{j_s} + \pi \text{ and } \theta_{j_1}+\pi  \le  \theta_{j_r}, &
\text{ where } j_s\in J_s.
\end{split}
\end{equation}
Recall that we drop the empty sets $I_s\cap J$ from $\sigma_J$ and shift the indexing. Thus the inequalities in the third line of \eqref{eq:octa_partial}, which define the boundary of the partial coamoeba, are generally stronger than the ones in \eqref{eq:octa}.

The {\bf alcove} decomposition of $\cc^n$ (the name comes from the affine root system $\hat A_n$) is the restriction of the triangulation of $\TT^n$ induced from the  decomposition of $\R^{n+1}/\R$ by  
the hyperplanes
\begin{equation}\label{eq:alcove}
\theta_i - \theta_j \in \pi \Z, \text{ for all pairs } i,j \in \hat n.
\end{equation}
The octahedra and their intersections with partial coamoebas are cut out by hyperplanes of the same form which means that all $\O_{\sigma,J}$ are triangulated by alcoves.

\begin{figure}[h]
\centering
\includegraphics[height=30mm]{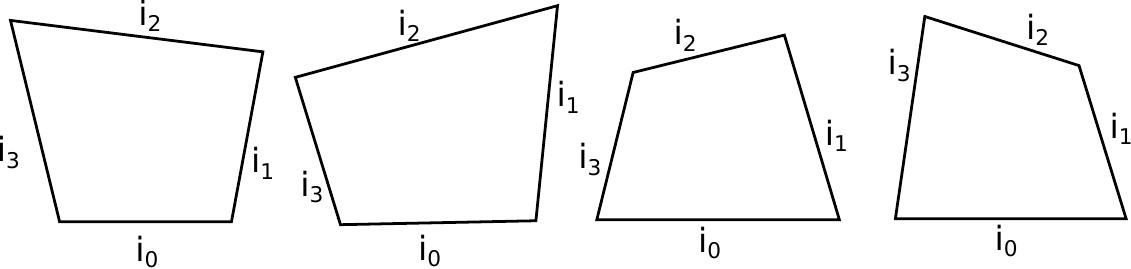}
\caption{The four shapes of generic quadrilaterals.}
\label{fig:alcove4}
\end{figure}

The hyperplanes \eqref{eq:alcove} break $\T^n$ into $n! \cdot 2^n$ maximal simplices. The $(n+1)!$ of them are incident to 0, they form the zonotope $Z$ and are not part of the coamoeba. Thus each of the $n!$ maximal octahedra in $\cc^n$ consists of $(2^n-n-1)$ maximal alcoves. 
For example, for $n=3$ each of the maximal octahedra is broken into 4 alcoves according to the relative directions of the opposite pairs of edges in the representing polygons (see Fig. \ref{fig:alcove4}).
The zero-dimensional octahedra are also the zero-dimensional alcoves, and they are the vertices of the coamoeba. 

To label alcoves we introduce 
certain combinatorial objects $\tau$ which refine cyclic partitions~$\sigma$. Think of $\sigma$ as coming from a configuration of points $\theta_0, \dots, \theta_n$ on the circle. Identifying the opposite points of the circle gives a new configuration of points on the quotient circle, that is, another cyclic partition $\tilde \sigma$. The purpose of $\tau$ is to encode both $\sigma$ and~$\tilde \sigma$.

Given a cyclic partition $\sigma = \<I_1, \dots, I_{k}\>$ we mark $k$ distinct points on the boundary of a disk (which we will call vertices) and label the $k$ boundary arcs between the vertices by the sets $I_s$ in the order given by $\sigma$. We say that a non-empty collection of chords in the disk with end points at the marked vertices is a {\bf net} $\tau$ if any two chords intersect (possibly at the end points).
If some of the vertices on the circle are not used by any of the chords in $\tau$ we can join the non-separated arcs together, thus getting a coarsening of $\sigma$ which we denote by $\sigma(\tau)$. Instead of the original~$\sigma$ we rather let the cyclic partition $\sigma(\tau)$ be a part of the intrinsic information in $\tau$.

\begin{figure}[h]
\centering
\includegraphics[height=30mm]{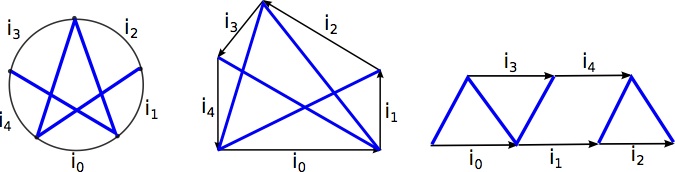}
\caption{Net of chords, diameters in $\mathcal D$ (sides $i_1$ and $i_4$ are parallel) and a subdivision of the M\"obius band.}
\label{fig:diagonals}
\end{figure}

One can think of nets as the M\"obius band decompositions as follows. We put the disk in $\R\PP^2$ and extend chords in $\tau$ to lines in $\R\PP^2$. The complement of the disc is the M\"obius band with  boundary broken into arcs also labeled by $\sigma(\tau)$. Any two chords intersect inside the disk, which means that their complements, which we call {\bf intervals}, do not intersect. That is, $\tau$ can be thought of a decomposition of the M\"obius band by intervals into triangles and trapezoids. Maximal decompositions with fixed $\sigma=\sigma(\tau)$ are triangulations (no trapezoids) and they use $k$ intervals. Minimal decompositions with fixed $\sigma=\sigma(\tau)$ consist of trapezoids (plus one triangle if $k$ is odd). Thus the number $l$ of chords in a net can be any integer between $k/2$ and $k$. 

The midcircle in the M\"obius band (the ``horizon'' in $\R\PP^2$) oriented along its boundary defines a new cyclic partition $\tilde \sigma(\tau)=\<K_1, \dots, K_l\>$ of $\hat n$, which we call the {\bf shuffle} of $\tau$.  Opposite sides of trapezoids in $\tau$ are combined into single subsets in $\tilde \sigma(\tau)$. The number of subsets in $ \tilde \sigma(\tau)$ equals number of chords in $\tau$.

\begin{example}\label{ex:chords}
The left picture in Figure \ref{fig:diagonals} is an example of a 4-chord net $\tau$ with $\sigma(\tau)=\<i_0, \dots, i_4\>$. The corresponding M\"obius band on the right is glued by identifying left and right blue intervals (turning one of the them upside down). We can always cut the M\"obius band along an interval in $\tau$ and picture the subdivision like that. 
This helps to see the shuffle order, which in this case is $\tilde\sigma(\tau) =\<i_0,  i_3,  \{i_1, i_4\}, i_2 \>$. 
\end{example}

Given a net $\tau$ we define the alcove $\am_\tau\subset \TT^n$ as the image of a simplex in $\R^{n+1}/\R$. As before we choose an initial set in $\sigma(\tau)= \<I_1, \dots, I_{k}\>$.
First, for $i,i'\in I_s$ we set 
\begin{equation}\label{eq:sameset}
\theta_i=\theta_{i'} =: \theta_{I_s}. 
\end{equation}
Next if $I_s$ and $I_r$ are opposite sides of a trapezoid in $\tau$ we set 
\begin{equation}\label{eq:opposite}
\theta_{I_r}+\pi = \theta_{I_s} \text{ if } r<s.
\end{equation}
Finally, we describe the inequalities, one for each chord (or interval) in $\tau$. Let $I_{s}\subseteq K_{s'}$ follow right after $I_{r}\subseteq K_{r'}$ in the shuffle order (that is, $s'=r'+1$ or $r'=l, \ s'=1$).
If $I_s$ also follows right after $I_r$ in the $\sigma(\tau)$-order (that is $s=r+1$ or $r=k, \ s=1$) we set:
\begin{equation}\label{eq:median}
\begin{split}
\theta_{I_r}  \le \theta_{I_{s}}, & \text{ if } 1\le r <k, \text{ or }\\
\theta_{I_r}  \le \theta_{I_s}+2\pi, & \text{ if } r=k,  s=1.
\end{split}
\end{equation}
If $I_s$ does not follow $I_r$ in the $\sigma(\tau)$-order we set:
\begin{equation}\label{eq:diagonal}
\begin{split}
\theta_{I_r} + \pi  \le \theta_{I_s}, & \text{ if } r<s, \text{ or }\\
\theta_{I_r}  \le \theta_{I_s}+\pi,  & \text{ if } r>s.
\end{split}
\end{equation}
If $K_{r'}$ or $K_{s'}$ (or both) contain more than one (i.e., two) subsets from $\sigma(\tau)$ then, given \eqref{eq:opposite}, any choice of a pair $(I_{r}, I_{s}) \subseteq (K_{r'}, K_{s'})$ gives rise to the same inequality. 

Altogether the inequalities \eqref{eq:sameset}, \eqref{eq:opposite}, \eqref{eq:median} and \eqref{eq:diagonal} define an $(l-1)$-dimensional simplex in $\R^{n+1}/\R$  which descends to a simplex in $\T^n$. This is the alcove $\am_\tau$. Its relative interior is defined by replacing \eqref{eq:median} and \eqref{eq:diagonal} with strict inequalities.
In Example \ref{ex:chords} (see Fig. \ref{fig:diagonals}) the alcove $\am_\tau$ is defined by 
$$ \theta_{i_4}= \theta_{i_1}+\pi, \quad \theta_{i_0}+\pi \le \theta_{i_3}, \quad \theta_{i_3}\le\theta_{i_4}, \quad 
\theta_{i_1} \le \theta_{i_2}, \quad \theta_{i_2} \le \theta_{i_0} +\pi.
$$

There is another,  non-minimal, but more intuitive set of inequalities which defines $\am_\tau$ directly in $\TT^n$. It keeps track of relative positions of {\em all} pairs of points $\theta_i, \theta_j$ on the circle, i.e., which half of the circle the differences $\theta_j-\theta_i$ belong to.
If $i,j$ are elements in different subsets in $\sigma(\tau)$ then any chord in $\tau$ which does {\em not} divides $\{i,j\}$ defines an order between $i$ and $j$ (going counter clockwise). All such chords in $\tau$ give the same order (otherwise, they would not intersect). We write $i\to_\tau j$ if if $i$ comes first in this order.
Two elements  $i,j\in \hat n$ may be divided by: 
\begin{enumerate}
\item no chords in $\tau$, that is $i,j$ belongs to the same $I_s$ in $\sigma(\tau)$,
\item all chords in $\tau$, that is $i,j$ lie in opposite sides of a trapezoid, or
\item some but not all chords in $\tau$, that is $i,j$ belong to different cells in the M\"obius band decomposition. 
\end{enumerate}
We set 
\begin{equation}\label{eq:circle}
\begin{split}
\theta_i=\theta_j, &\text{ in case (1)}\\
\theta_i - \theta_j = \pi \mod 2\pi, & \text{ in case (2)},\\
\theta_j -\theta_i \in [0,\pi] \mod 2\pi, &  \text{ in case (3) with }  i\to_\tau j.
\end{split}
\end{equation}

\begin{lemma}\label{lemma:alcoves}
The alcove $\am_\tau$ is defined by \eqref{eq:circle}.
\end{lemma}
\begin{proof}
In the lift to $\R^{n+1}/\R$ associated with the initial subset $I_1$ in $\sigma(\tau)$ the first two equations in \eqref{eq:circle} are the same as \eqref{eq:sameset} and \eqref{eq:opposite}. The inequalities \eqref{eq:median} and \eqref{eq:diagonal} form a subset of the third line in \eqref{eq:circle} for pairs $i,j$ next to each other in the shuffle order. 
 
In the opposite direction, let us deduce, say, $\theta_{I_r} \le \theta_{I_s}$ for $I_r \to_\tau I_s$ and  $r<s$. Other inequalities in \eqref{eq:circle} are similar. Choose a chord which does not divide $I_r$ and $I_s$ and cut the M\"obius band along it. Then the M\"obius band unfolds into a strip with $I_r$ and $I_s$ on one side. By induction, we may assume that $I_r$ and $I_s$ are neighbors in the $\sigma(\tau)$-order, that is $s=r+1$, but there may be several subsets $I_{r'}, \dots, I_{r''}$ ``shuffled'' in-between in the shuffle order, see Fig. \ref{fig:medians}. 

\begin{figure}[h]
\centering
\includegraphics[height=20mm]{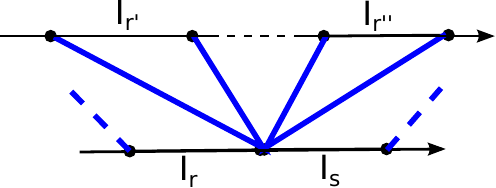}
\caption{A fragment of the M\"obius band.}
\label{fig:medians}
\end{figure}

There are several cases for the order among the subscripts $r, r', r''$. We consider, e.g., the case $r'' < r < s =r+1 < r'$, others are similar. Then \eqref{eq:median} and \eqref{eq:diagonal} will read:
\begin{equation}
\begin{split}
\pi+ \theta_{I_r}\le \theta_{I_{r'}} \le \dots \le \theta_{I_{k}} \le \theta_{I_{1}} +2\pi \le \dots \le & \theta_{I_{r''}} + 2\pi,\\
& \theta_{I_{r''}} + \pi \le \theta_{I_s},
\end{split}
\end{equation}
which imply $\theta_{I_r} \le \theta_{I_s}$.
\end{proof}

All nets, or equivalently all subdivisions of the M\"obius band, form a poset under refinements. We set $\rk \tau := l-1$, where $l$ is number of chords in $\tau$. This is the dimension of the alcove $\am_\tau$. Clearly, $\tau\preceq \tau'\Rightarrow \sigma(\tau)\preceq \sigma(\tau')$. For any $J\subseteq \hat n$ we say a chord in $\tau$ {\bf divides} $J$ if $J$ does not lie on one side of it. We say $\tau$ {\bf divides} $J$ (and write $\tau | J$) if all of its chords do. In Example \ref{ex:chords} (see Fig. \ref{fig:diagonals}) the net $\tau$ divides only $J=\{i_1, i_4\}$ among all two element subsets, but it divides any $J\subseteq \hat n$ with $|J|\ge 3$.

\begin{proposition}\label{prop:alcoves}
$\am_\tau\subseteq \O_{\sigma,J}$ if and only if $\sigma(\tau)\preceq \sigma$ and $\tau$ divides $J$. \end{proposition}

\begin{proof}
The ``if'' part follows directly from Lemma \ref{lemma:alcoves}. Indeed, the inequalities \eqref{eq:octa_partial}, which define $\O_{\sigma, J}$, are special cases of  \eqref{eq:circle} for $i,j$ belonging to neighboring subsets in $\sigma(\tau)$.

For  the converse, to a polygon $\mathcal D$ representing a point in $\O_{\sigma,J}$ we associate a net $\tau$ with $\sigma(\tau)\preceq \sigma$ and $\tau | J$ as follows. 
Given a line through a vertex $v$ of $\mathcal D$ we say that $\mathcal D$ lies {\em strictly} on one side of the line if both adjacent edge vectors from $v$ lie in the same {\em open} half-plane. Two vertices are connected by a {\bf diameter} if $\mathcal D$ lies strictly between two parallel lines through these two vertices.  
Any polygon $\mathcal D$ has at least one diameter, a geometric one, and 
clearly, any two diameters intersect. That is, the set of diameters forms a net $\tau$ on a disk with boundary arcs labeled by the sides of $\mathcal D$ (see Fig. \ref{fig:diagonals}). Moreover any diameter must have non-zero edges on both sides of it, that is $\tau$ divides $J$.

A set of diameters in a polygon $\mathcal D$ recovers all pair-wise relations among the directions~$\theta_i$ of its sides. It is easy to see that these relations are given by  \eqref{eq:circle} for the corresponding net $\tau$. Thus the point in $\O_{\sigma,J}$ represented by $\mathcal D$ falls into the alcove~$\am_\tau$ in $\cc^n$.
\end{proof}

\subsection{Phase tropical pair-of-pants as a CW complex} 
Consider
$$F(x)= \max \{x_0,x_1,\dots , x_n\},
$$ 
a convex PL function on $\R^{n+1}$.
Its corner locus is invariant under the diagonal translation by $\R$, hence it descends to an $(n-1)$-dimensional polyhedral fan in $\R^{n+1}/\R$, which is known as the {\bf tropical hyperplane} $\pp^{n-1}$.
The cones $\pp_I$ in $\pp^{n-1}$ are indexed by subsets $I\subseteq \hat n$ of size $|I|\ge 2$:  the cone $\pp_I$ is defined by
$$ x_i=x_j \ge x_k, \text{ for all } i, j\in I, \ k\not\in I.
$$
The vertex $\pp_{\hat n}$ of $\pp^{n-1}$ is at the origin in $\R^{n+1}/\R$.

The closure $\bar\pp^{n-1}$ of $\pp^{n-1}$ in $\Delta$ is a polyhedral complex: the map \eqref{eq:moment} takes each cone $\pp_I$ to a linear subspace in the interior of $\Delta$ and the linearity extends to the closure. Also note that the face $\Delta_{I'}$ of the simplex $\Delta$ intersects the closure of $\pp_I$ only if the subset $I'$ contains~$I$, which gives additional labeling to the boundary faces of $\bar\pp^{n-1}$ by subsets  $I'\supseteq I$. Each face $\pp_{I, I'}$ of $\bar\pp^{n-1}$ is a polytope of dimension $|I'|-|I|$.

The {\bf phase tropical pair-of-pants} $\tt\pp^{n-1} \subset (\C^*)^{n+1}/\C^* = (\R^{n+1}/\R) \times \TT^n$ is the union
$$\tt\pp^{n-1}: = \bigcup_{I\subseteq \hat n} (\pp_I \times \cc_I).
$$
The compactified version $\tt \bar \pp^{n-1}$ is the closure of $\tt\pp^{n-1}$ in  $\Delta \times \TT^n$.
The $(\sigma,J)$-stratification on $\Delta \times \T^n$ induces a stratification of $\tt \bar \pp^{n-1}$. We denote by $\bar \Psi_{\sigma, J}$ the corresponding closed stratum of $\tt \bar \pp^{n-1}$. 

Proposition \ref{prop:alcoves} says that each $\O_{\sigma,I}$ is triangulated into alcoves $\am_\tau$ with $\sigma(\tau)\preceq\sigma$  and $\tau | I$. This makes $\bar\Psi_{\sigma, J}$ into a polyhedral complex in $\Delta \times \TT^n$:
$$\bar\Psi_{\sigma, J}= \bigcup_{(I, I', \tau)}  \pp_{I,I'} \times \am_{\tau},
$$
where the triples $(I, I', \tau)$ satisfy $I \subseteq I' \subseteq J$, $\sigma(\tau) \preceq \sigma$ and $\tau | I$. The face order between legitimate triples is $(I, I', \tau)\preceq (\tilde I, \tilde I', \tilde \tau)$ if $I\supseteq \tilde I$, $I'\subseteq \tilde I'$ and $\tau\preceq \tilde \tau$.

\begin{proposition}\label{prop:tropicalCW}
The decomposition of $\tt\bar\pp^{n-1}$ into $ \bar \Psi_{\sigma,J}$ is a regular CW complex.
\end{proposition}
\begin{proof}
Lemmas \ref{lemma:tropicalCW} and \ref{lemma:Psi} below show that each $\bar\Psi_{\sigma,J}$ is a collapsible PL manifold with boundary of dimension $\rk(\sigma,J)$. Then a version of the regular neighborhood theorem (cf., e.g., \cite{Forman}, Theorem 1.6) implies that $\bar \Psi_{\sigma,J}$ is homeomorphic to the closed ball of dimension $\rk(\sigma,J)$. 
\end{proof}

We begin with the collapsibility. 
Recall the collapsing operation on a polyhedral complex $X$. Let $F$ be a face of $X$ and let $G$ be a facet of $F$, such that $G$ is not a subface of any other face in $X$. Then we can remove both $F$ and $G$ and call this an elementary collapse.    
We say that a polyhedral complex is {\bf collapsible} if it can be reduced to a vertex by a sequence of elementary collapses.

\begin{lemma}\label{lemma:tropicalCW}
$\bar\Psi_{\sigma,J}$ is a collapsible polyhedral complex of pure dimension $\rk(\sigma,J)$.
\end{lemma}

\begin{proof}
The maximal faces $(I, I', \tau)$ in $\bar\Psi_{\sigma,J}$ are of two types. Type I: $\tau$ is a maximal net with $\sigma(\tau)=\sigma$ and $I$ has three elements (a maximal $\tau$ cannot divide a set of two elements). Type II: $\tau$ contains a single trapezoid (and $k-2$ triangles) and $|I|=2$, its two elements belong to the opposite sides of the trapezoid.
In both cases $I'=J$ and 
$$\dim ( \pp_{I,J} \times \am_{I(\tau)}) = |J| -|I| + (\# \{\text{chords in } \tau\} -1) = |J| + k -4 = \rk (\sigma, J).
$$
Any face $(I, I', \tau)$ is a subface of $(I, J, \tau)$. Adding chords to $\tau$ and/or removing elements from $I$ one can see that any face is a subface of a maximal face. That is, the complex is indeed of pure dimension $\rk(\sigma,J)$.

To collapse $\bar \Psi_{\sigma,J}$ we look at its face lattice. 
For a given pair $(I, \tau)$ the interval between $(I, I, \tau)$ and $(I, J, \tau)$ consists of all subsets $I'$ between $I$ and $J$. In particular, it is Boolean, hence possesses a matching unless $I=J$. 
Note that if $\tau$ is maximal, elements in the interval $[(I, I, \tau),(I, J, \tau)]$ are not subfaces of anything outside the interval. Then we can remove the entire interval. The same can be said about intervals with $|I|=2$. Proceeding by alternating  induction on number of chords in $\tau$ and number of elements in $I$ we remove all faces of $\bar \Psi_{\sigma,J}$ except those with $I=J$.

Thus, it remains to collapse the fiber over the vertex  $\pp_{J, J}$. This fiber is the polytope $\O_{\sigma,J}$ triangulated into alcoves which is clearly collapsible.
\end{proof}

\begin{lemma}\label{lemma:Psi}
$\bar\Psi_{\sigma,J}$ is a topological manifold with boundary.
\end{lemma}

Before proving Lemma \ref{lemma:Psi} we discuss a certain general property of convex cones.
Let $V$ be an $n$-dimensional vector space over $\mathbb{R}$ and write $V^*$ for its dual.  Let $\mathcal R \subset V$ and $\RR^\vee \subset V^*$ be a dual pair of convex cones. That is
\begin{align}
\mathcal{R} = \{v \in V \suchthat \lambda (v) \geq 0, \lambda \in \RR^\vee \},
\end{align}
and vice versa. We assume both $\RR$ and $\RR^\vee$ have non-empty interiors. For any $v \in \mathcal{R}$ we define the {\bf supporting tangent cone} $T_v \mathcal{R}$ to be the set of vectors in $V$ lying in supporting hyperplanes to $\mathcal R$ at $v$.

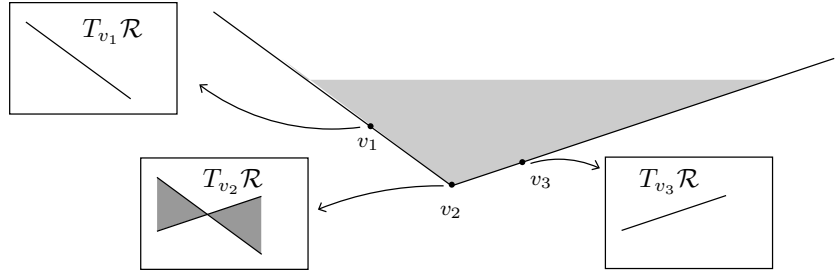
\begin{figure}[h]
	\begin{equation*}
	\begin{tikzpicture}[cross line/.style={preaction={draw=white, -, line width=6pt}}]
	\node[anchor=south west,inner sep=0] (image) at (0,0) {\includegraphics{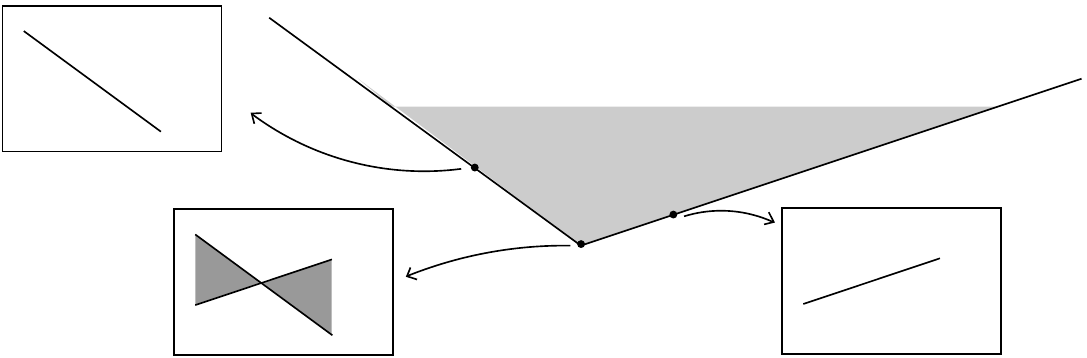}};
	\node[align=center] at (4.8,1.7) {\tiny $v_1$};
	\node[align=center] at (5.9,.8) {\tiny $v_2$};
	\node[align=center] at (1.4,3.2) {\footnotesize $T_{v_1} \mathcal{R}$};
	\node[align=center] at (7.1,1.2) {\tiny $v_3$};
	\node[align=center] at (3,1.2) {\footnotesize $T_{v_2} \mathcal{R}$};
	\node[align=center] at (8.8,1.2) {\footnotesize $T_{v_3} \mathcal{R}$};
	\end{tikzpicture} 
	\end{equation*}
	\caption{\label{fig:tangents} Two dimensional cone and fibers of the total tangent cone of its boundary.}
\end{figure}

We define the {\bf total supporting tangent space} of $\mathcal{R}$ to be the union 
\begin{align}\label{eq:TR}
T \mathcal{R} := \bigcup_{v \in \mathcal{R}} (\{v\} \times T_v \mathcal{R} ) = \bigcup_{\lambda\in \RR^\vee \setminus \{0\}} ((\ker \lambda \cap \RR) \times \ker \lambda) \subset V \times V.
\end{align}  
The first presentation shows that $T \mathcal{R} $ is a fibration over $\mathcal R$ supported on its boundary $\partial R$. An example of a two dimensional polyhedral cone $\mathcal{R}$ is illustrated in Figure~\ref{fig:tangents}.

We fix a vector $\tilde{v}$ in the interior of $\mathcal{R}$ and a vector $\tilde \lambda$ in the interior of $\RR^\vee$, such that $\tilde\lambda (\tilde{v}) =1$. Then in the second presentation of $T\RR$ in \eqref{eq:TR} it is enough to take the union over $\{\lambda \in \RR^\vee \suchthat \lambda(\tilde v) =1\}$. 

Denote by $W$ the quotient space $V/(\R\tilde v)$, and we  write $\pi : V \to W$ for the projection.  Consider the map $\phi : T \mathcal{R} \to W$ given by $\phi (v, u) = \pi(u)$.

\begin{lemma}\label{lemma:convex}
The total supporting tangent space $T \mathcal{R}$ is homeomorphic to $\mathbb{R}^{2n - 2}$. Moreover, the map $\phi :  T \mathcal{R} \to W$ is a trivial fiber bundle with fiber homeomorphic to $W\cong \R^{n-1}$.
\end{lemma}

\begin{proof}
We will show that the following map $\psi :T \mathcal{R} \to W \times W $ given by 
	\begin{align} \label{eq:homeo}
	\psi ( v, u ) = \left( \pi(v) + \tilde \lambda( u) \cdot \pi (u), \pi(u) \right).
	\end{align}
is a homeomorphism. The geometric meaning of the map $\psi$ is to ``stretch out'' the fibers $\phi^{-1} (w)$ into $\R^{n-1}$, see Figure \ref{fig:dirder}.
\begin{figure}[h]
	\includegraphics[scale=.25]{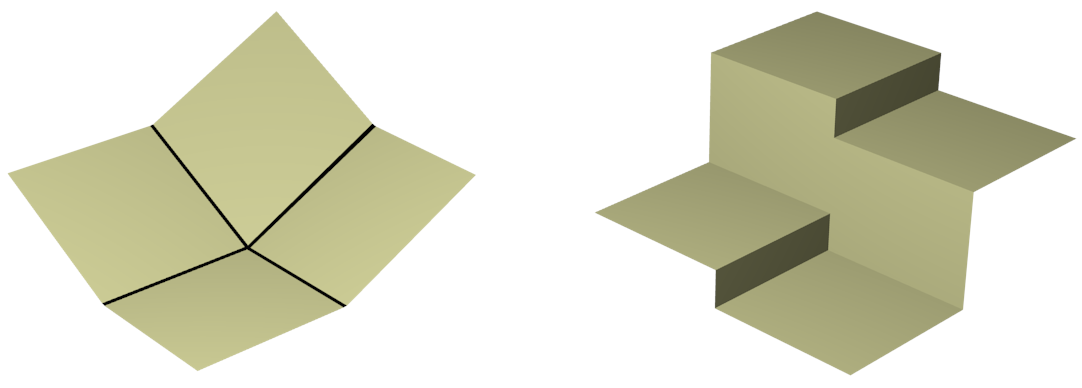}
	\caption{\label{fig:dirder} The boundary $\partial \RR$ on the left and an example of the fiber  $\phi^{-1} (w)$ for some $w\in W$. $\phi^{-1} (w)$ is homeomorphic to $\R^{2}$.}
\end{figure} 
Then $\phi : T  \mathcal{R} \to W$ is the composition of $\psi$ with the projection onto the second factor and, hence, is a topologically trivial fiber bundle with fiber homeomorphic to $W$.

The map $\psi$ is clearly continuous. Surjectivity follows from the Intermediate Value theorem applied to each fiber $\phi^{-1} (w)$ for a given $w\in W$. 

Injectivity: Let $\psi ( v_1, u_1 )= \psi ( v_2, u_2 )$. Then according to \eqref{eq:homeo} we must have 
\begin{equation}\label{eq:v=v}
\begin{split}
v_1-v_2 + \tilde \lambda(u_1-u_2) \cdot u_{1} &= 0 \mod \tilde v\\
v_1-v_2 + \tilde \lambda(u_1-u_2) \cdot u_{2} &= 0 \mod \tilde v.
\end{split}
\end{equation}
Say, $v_1, u_1 \in \ker \lambda_1$ and $v_2, u_2 \in \ker \lambda_2$ for some $\lambda_{1,2} \in \RR^\vee$ with $\lambda_{1,2}(\tilde v)=1$. Note that $\lambda_1(v_2)\ge 0$ and $\lambda_2(v_1)\ge 0$. Applying $\lambda_1$ and $\lambda_2$, to the respective equations \eqref{eq:v=v} we get 
\begin{align*}
v_1-v_2 + \tilde \lambda(u_1-u_2) \cdot u_{1} &= -\lambda_1(v_2) \cdot \tilde v\\
v_1-v_2 + \tilde \lambda(u_1-u_2) \cdot u_{2} &= \lambda_2(v_1) \cdot \tilde v
\end{align*}
Subtracting one equation from another we get 
\begin{equation}\label{eq:u=u}
 \tilde \lambda(u_1-u_2) \cdot (u_{1}-u_2) = - (\lambda_1(v_2) +\lambda_2(v_1)) \cdot \tilde v.
\end{equation}
Finally applying $\tilde \lambda$ we arrive at $\tilde \lambda(u_1-u_2)^2 \le 0$, which combined with $\pi(u_1)=\pi(u_2)$ gives $u_1=u_2$. That implies that $\pi(v_1)=\pi (v_2)$, and since both $v_{1,2}\in \partial \RR$, we have $v_1=v_2$.

Continuity of the inverse: fix a Euclidean metric on $ \ker \tilde \lambda \cong W$ and extend it to a Euclidean metric on $V$. Then notice that the linear functionals $\lambda\in \RR^\vee$ with $\lambda(\tilde v)=1$ are uniformly  bounded, and so are the ratios $|u|/|\pi(u)| , |v|/|\pi(v)|$ for all $(v, u) \in T\RR$.
Suppose $|\psi ( v_1, u_1 )- \psi ( v_2, u_2 )| < \epsilon$. Then following the injectivity arguments above we deduce that $|v_1-v_2| < C\epsilon$ and $| u_1- u_2| < C\epsilon$ for some universal constant $C$.
\end{proof}

\begin{remark}
The lemma is a consequence of the convexity property of $\RR$. It holds for $\RR$ replaced by the upper graph of any convex function $f: W \to \R$.
\end{remark}

For the proof of Lemma \ref{lemma:Psi} we will apply the above lemma to the following convex cone. Let 
$\sigma_0=\<I_-, I_+\>
$ 
be a cyclic 2-partition of~$\hat n$ together with a choice of the initial subset. Define a convex polyhedral cone $\mathcal R \subset \R^{n+1}/\R$ by the following set of inequalities:
\begin{equation}\label{eq:convex}
x_{i_-} \le x_{i_+}, \quad  i_-\in I_-, i_+ \in I_+.
\end{equation}
Its boundary $\partial \mathcal R$ is a polyhedral fan, whose cones $\mathcal R_{I}$ are labeled by 
subsets $I\subseteq \hat n$ which are divided by $\sigma_0$, i.e. both $I_-\cap I$ and $I_+\cap I$ are non-empty. That is, the face lattice of $\partial \mathcal R$ is dual to the face lattice of the product of simplices $ \Delta_{I_-} \times \Delta_{I_+}$.
The cone $\mathcal R_{I}$ is defined by 
\begin{equation}\label{eq:face_convex}
\begin{split}
x_i = x_{i'}, & \text{ for } i, i' \in I,\\
x_{i_-} \le x_{i_+}, & \text{ for } i_\pm\in I_\pm.
\end{split}
\end{equation}

We identify the tangent space of $\R^{n+1}/\R$ at any point with $\R^{n+1}/\R$ and consider the total supporting tangent space:
$$T \mathcal R \subset (\R^{n+1}/\R) \times (\R^{n+1}/\R).
$$
Let $[u_0, \dots, u_n]$ be the homogeneous coordinates in the second (tangent) factor which are parallel to $[x_0, \dots, x_n]$. Following the notation introduced just before Lemma \ref{lemma:convex}, we let $\tilde v \in \R^{n+1}/\R$ be the vector with coordinates 
\begin{equation}\label{eq:tilde_v}
\begin{cases}
u_{i_-} = 0,\quad i_- \in I_- \\
u_{i_+} = 1, \quad i_+ \in I_+
\end{cases}
\end{equation}
and denote by $W$ the quotient $(\R^{n+1}/\R)/(\R\tilde v)$. Let $\phi: T \mathcal R \to  W$ be the projection onto $W$ in the tangent factor.

Now let
$$\sigma=\< \underbrace{I_1, \dots, I_{r}}_{I_-}, \underbrace{I_{r+1}, \dots, I_k}_{I_+}\>,
$$ 
be a refinement of $\sigma_0$.
The following set of inequalities 
\begin{equation}\label{eq:C_v}
\begin{split}
u_i=u_{i'} =: u_{I_s}, &  \text{ for } i, i' \in I_s, \\
u_{I_1}\le \dots \le u_{I_{r}}, & \  u_{I_{r+1}} \le \dots  \le u_{I_k},
\end{split}
\end{equation}
defines a cone $C_{\sigma_0,\sigma}$ in $W=(\R^{n+1}/\R)/(\R\tilde v)$. Then by Lemma~\ref{lemma:convex}, $\phi^{-1}(C_{\sigma_0, \sigma})$ is homeomorphic to $W \times C_{\sigma_0,\sigma}$, which is a manifold with boundary. 

Next we describe the fibers of the projection of $\phi^{-1}(C_{\sigma_0, \sigma})\subset \partial \mathcal R \times (\R^{n+1}/\R)$ onto the first factor. Let $x$ be a point in the relative interior of a face $\mathcal R_{I}$ of $\partial \mathcal R$. A vector $u \in \R^{n+1}/\R$ is in the supporting tangent cone $T_x \mathcal R$ if and only if it is in the kernel of some non-zero linear functional $\lambda\in\RR^\vee$ with $\lambda(x)=0$. But $\RR^\vee \cap (\ker x)$ is positively spanned by the vectors $e_{i_+} - e_{i_-}, \ i_\pm \in I_\pm \cap I$, where $\{e_i\}$ is the dual basis to the coordinates $[x_0, \dots, x_n]$. It means that there are two pairs $i^-_{1,2} \in {I_-} \cap I$ and $i^+_{1, 2} \in {I_+}\cap I$ such that 
\begin{equation}\label{eq:wedge}
u_{i^-_{1}}- u_{i^+_1} \le 0  \text{ and } u_{i^-_2}- u_{i^+_2} \ge 0.
\end{equation}

Consider the partition of $I$ induced by $\sigma$:
\begin{align}
\sigma_I= &\< \underbrace{I^-_{\min},\dots,  I^-_{\max}}_{I_- \cap I}, \underbrace{I^+_{\min}, \dots, I^+_{\max}}_{I_+ \cap I}\>.
\end{align}
Then given the inequalities \eqref{eq:C_v} for the cone $C_{\sigma_0,\sigma}$, the existence of pairs $i^-_{1,2} \in {I_-} \cap I$ and $i^+_{1, 2} \in {I_+}\cap I$ satisfying \eqref{eq:wedge} becomes equivalent to 
\begin{equation}\label{eq:graph_fibers}
 u_{i^-_{\max}} \ge u_{i^+_{\min}}  \text{ and }   u_{i^+_{\max}} \ge u_{i^-_{\min}}, \ i^\pm_{\min, \max} \in I^\pm_{\min, \max}. 
\end{equation}
Thus, the fiber of $\phi^{-1}(C_{\sigma_0,\sigma})$ over the relative interior of a face $\mathcal R_{I}$  is cut out by the following set of inequalities:
\begin{equation}\label{eq:phi^{-1}(C_v)}
\begin{split}
u_i  =u_{i'} =: u_{I_s}, & \text{ for }  i, i' \in I_s \\
u_{I_1}\le \dots \le u_{I_{r}},  \  u_{I_{r+1}} \le \dots  \le u_{I_k}, &\\
  u_{i^-_{\max}} \ge u_{i^+_{\min}} \text{ and }  u_{i^+_{\max}} \ge u_{i^-_{\min}}, & \text{ where } i^\pm_{\min, \max} \in I^\pm_{\min, \max}.
\end{split}
\end{equation}

Finally, we remark that up to an isomorphism the space $\phi^{-1}(C_{\sigma_0,\sigma})$ does not depend on the choice of the initial subset in $\sigma_0$. Had we chosen $I_+$ instead of $I_-$, the cone $\mathcal R$ would have changed to its opposite. The fibers over the corresponding cones in $\partial \mathcal R$ would remain the same.

Let us return to $\tt\bar\pp^{n-1}$. First, recall the notion of the local fan in a  polyhedral complex. If $v$ is a vertex in a face $F$ of a polyhedral complex $X\subset \R^n$, then the {\bf local cone} $\Sigma_v F$ of $F$ at $v$ is the set of vectors 
\begin{equation}\label{eq:rel_cone}
\{w\in T_v\R^n \suchthat v+\epsilon w \in F,  \text{ for some } \epsilon>0\}
\end{equation}
in the tangent space $T_v\R^n\cong \R^n$. The {\bf local fan} $\Sigma_v X$ of $X$ at $v$ is the union of local cones over all faces $F$ of $X$ containing $v$.

\begin{proof}[Proof of Lemma \ref{lemma:Psi}]
We will show that the local fan $\Sigma_v \bar \Psi_{\sigma, J}$ at any vertex $v$ of $\bar \Psi_{\sigma, J}$ is homeomorphic to the $(k+|J|-4)$-dimensional half-space. First, we consider a maximal case: $J=\hat n$ and $v$ is a vertex of $\bar \Psi_{\sigma}$ which lies over the vertex of  $\pp^{n-1}$. That is, $v=\pp_{\bar n} \times \{a\}$, where $a=\pi_{I_-}=\pi_{I_+}$ (cf. Notations) is a vertex of $\O_{\sigma}$. The corresponding cyclic 2-partition 
$$\sigma_a=\<I_-, I_+\>=\<I_1\cup \dots \cup I_{r}, I_{r+1} \cup \dots \cup I_k \>
$$ 
is a coarsening of $\sigma$. 

The local fan $\Sigma_v \bar \Psi_{\sigma}$ maps to the tropical hyperplane $\pp^{n-1}$. The fiber over the relative interior of a cone $\pp_I$ is non-empty if and only if $a \in \O_{\sigma, I}$, that is, if both sets $I\cap I_\pm$ are non-empty. Thus, the collection of faces $\pp_I$ with non-empty fibers forms a subfan $\pp(a)$ of $\pp^{n-1}$ whose 
face lattice is dual to the face lattice of $\Delta_{I_-} \times \Delta_{I_+}$. In particular the fan $\pp(a)$ is isomorphic to the boundary fan $\partial \mathcal R$ of the polyhedral cone $\mathcal R$ defined in \eqref{eq:convex}.

Next we describe the fiber over the relative interior of $\pp_I$. It is the relative cone of the polytope $\O_{\sigma, I}$ at its vertex $a$. Let $u_i$ be the homogeneous coordinates on the tangent space $T_v \TT^n=\R^{n+1}/\R$ which are parallel to the coordinates $\theta_i$ on $\TT^n$, and let 
$$\sigma_I = \<\underbrace{I^-_{\min},\dots,  I^-_{\max}}_{I\cap I_-}, \underbrace{ I^+_{\min}, \dots, I^+_{\max}}_{I\cap I_-}\>
$$
be the cyclic partition of $I$ induced by $\sigma$. Then the subset of the defining 
inequalities~\eqref{eq:octa_partial} for the polytope $\O_{\sigma, I}$ at $a$ is 
identical to \eqref{eq:phi^{-1}(C_v)}. Thus $\Sigma_v \Psi_{\sigma}$ is homeomorphic to $\phi^{-1}(C_{\sigma_0,\sigma})$, which by Lemma \ref{lemma:convex} is homeomorphic to the $(n+k-3)$-dimensional half-space.

Now we allow $J$ to be a proper subset of $\hat n$, but still consider a vertex $v$ of $\bar \Psi_{\sigma, J}$ which lies over the vertex of the corresponding lower dimensional tropical hyperplane $\pp^{|J|-2}\subset~\Delta_J$. That is, $v=\mathcal P_{J,J} \times \{a\}$, where $a$ is a vertex of $\O_{\sigma,J}$.
Locally near $v$ the space $(\tt\bar \pp^{n-1}) \cap (\Delta_J \times \T^n)$ is the product $\tt \pp^{|J|-2} \times \T^{\hat n \setminus J}$. Moreover, by choosing the splitting $\R^{n+1}/\R = \R^J/\R \times  \R^{\hat n \setminus J}$  such that the vector $\tilde v$ (cf. \eqref{eq:tilde_v}) lies in  $\R^J/\R$  the product structure can be made compatible with the projection by $\tilde v$. Then the projection $\phi: \phi^{-1}(C_{\sigma_a, \sigma}) \to W$ is again a trivial fiber bundle with fibers homeomorphic to $(\R^J/\R) /(\R\tilde v)$ as in the maximal case for lower dimensional pair-of-pants.

Finally, for a non-central vertex $v=\mathcal P_{I,I} \times \{a\}$ of $\Psi_{\sigma, J}$, where $I\subsetneq J$, the local fan $\Sigma_v \Psi_{\sigma, J}$ is just the product $\Sigma_v \Psi_{\sigma, I} \times \R_{\ge 0}^{J\setminus I}$.
\end{proof}

\begin{remark}
Although the polyhedral fans  $\partial \mathcal  R$ and $\pp(v)$  are isomorphic, they are really different fans  in $\R^{n+1}/\R$. The former bounds a convex cone, the latter generally does not.
\end{remark}

For any vertex $a$ of the coamoeba $\cc^n$ one can identify the local fan of $\tt\pp^{n-1}$ at $v=\pp_{\bar n} \times \{a\}$ with the total supporting tangent space $T \mathcal  R$ of the cone $\mathcal R$ associated with the corresponding 2-partition $\sigma_a=\<I_-, I_+\>$, which is homeomorphic to a $(2n-2)$-ball. That directly proves a conjecture of Viro \cite[Section 5.10]{Viro} that $\tt\pp^{n-1}$ is a topological manifold.

\begin{theorem}\label{thm:homeo}
$\tt \bar \pp^{n-1}$ is homeomorphic to $\bar P^{n-1}$, and $\tt \pp^{n-1}$ is homeomorphic to $P^{n-1}$.
\end{theorem}
\begin{proof}
Two isomorphic regular CW complexes are homeomorphic (cf, e.g. \cite {Bjorner}). And so are their interiors.
\end{proof}

\begin{remark}
It may be desirable to extend the result to a homeomorphism of pairs 
$$(P^{n-1}, (\C^*)^{n+1}/\C) \approx (\tt \pp^{n-1}, (\C^*)^{n+1}/\C).$$ 
Indeed, the respective complements are the higher dimensional pairs-of-pants themselves. The problem, however, is that even the local homeomorphism $\Sigma_v \tt\pp^{n-1}\approx \R^{2n-2}$ given by \eqref{eq:homeo} is fairly complicated and it does not seem to have an obvious extension to a tubular neighborhood.
\end{remark}

\section{Phase tropical varieties of hypersurfaces in $(\mathbb{C}^*)^n$}
\label{sec:H}

\subsection{Phase tropical varieties} In this section, we will give the necessary background to state our main result, Theorem~\ref{thm:maintheorem}. We begin with some preliminary definitions which may be found in \cite{GKZ}.

Let $N \cong \mathbb{Z}^n$, $M = \textnormal{Hom} (N, \mathbb{Z})$ and, for any abelian group $\mathbb{K}$ write $M_\mathbb{K}$ for $M \otimes_{\mathbb{Z}} \mathbb{K}$ and similarly for $N$. Let $A$ be a finite subset of $M$ and $Q$ its convex hull in $M_\mathbb{R}$. We call $(Q, A)$ a \textbf{marked polytope}. Another marked polytope $(Q^\prime, A^\prime)$ will be called a \textbf{face} of $(Q,A)$ if $Q^\prime$ is a face of $Q$ and $A^\prime = A \cap Q^\prime$. If $A$ is affinely independent, we say it is a \textbf{marked simplex}. A subdivision $\mathcal{S} = \{(Q_\gamma, A_\gamma) : \gamma \in \Gamma\}$ of $(Q,A)$ is a collection of marked polytopes satisfying:
\begin{enumerate}
	\item for any $\gamma \in \Gamma$, every face of $(Q_\gamma, A_\gamma)$ is in $\mathcal{S}$,
	\item for any $\gamma, \tilde{\gamma} \in \Gamma$, the intersection $(Q_\gamma \cap Q_{\tilde{\gamma}} , A_\gamma \cap A_{\tilde{\gamma}})$ is in $\mathcal{S}$ and is a face of both $(Q_\gamma, A_\gamma)$ and $(Q_{\tilde{\gamma}} , A_{\tilde{\gamma}})$,
	\item the union $\cup_{\gamma \in \Gamma} Q_\gamma$ equals $Q$.
\end{enumerate}
The poset $\Gamma$ is the face lattice of the subdivision. Note that it is not necessarily the case that $\cup_{\gamma \in \Gamma} A_\gamma = A$. 

We will be particularly interested in subdivisions that are \textbf{coherent}. The basic ingredient in this construction is a function $\eta : A \to \mathbb{R}$. 
From $\eta$, we  define a polyhedron in $M_\mathbb{R} \times \mathbb{R}$ as the convex hull 
\begin{align}
\label{eq:polyhedron} \bar{Q}_\eta := \textnormal{Conv} \left\{ (\alpha, r) \in A \times \mathbb{R} : r \geq \eta (\alpha )  \right\}.
\end{align}
Let $\bar A_{\bar F}$ be the set of vertices over $A$ of any compact face $\bar F$ of $\bar{Q}_\eta$, and take $A_F, F$ to be their projections onto $M_\R$. Define the subdivision $\mathcal{S}_\eta = \{(F, A_F) : \bar{F} \text{ a compact face of } \bar{Q}_\eta\}$. When each $(F, A_F)$ is a marked simplex, we say that $\eta$ induces a \textbf{coherent triangulation} of $(Q, A)$ (see \cite[Chapter~7]{GKZ}). 

Consider the piecewise linear function $F_\eta : N_\mathbb{R} \to \mathbb{R}$ (or tropical polynomial) defined by
\begin{align}
F_\eta (\mathbf{x})  = \max \{ \alpha (\mathbf{x}) - \eta (\alpha ) : \alpha \in A \}.
\end{align}
The \textbf{tropical hypersurface} $\tropvar{\eta} \subset N_\mathbb{R}$ is the corner locus of $F_\eta$, see, e.g. \cite{MS} for details.
Note that $\tropvar{\eta}$ is a polyhedral complex. For any $k \in \mathbb{Z}$, let $\mathcal{S}_\eta^{\geq k}$ be the set of faces $(Q_\gamma, A_\gamma)$ for which $\dim Q_\gamma \geq k$. There is an order reversing bijection 
\begin{equation}
\label{diag:commtropcomp} 
\begin{tikzpicture}[baseline=(current  bounding  box.center), node distance=3cm, auto]
\node (A) {$\mathcal{S}_\eta^{\geq 1}$};
\node (B) [right of=A]{$\left\{ \text{Faces of }\tropvar{\eta} \right\}$}; 
\path[->,font=\scriptsize]
(A) edge node[above] {$\Phi$} (B);
\end{tikzpicture}
\end{equation}
where $\Phi (Q_\gamma, A_\gamma) = \{\mathbf{x} \in N_\mathbb{R}: \alpha (\mathbf{x}) - \eta (\alpha) = F_\eta (\mathbf{x}) , \text{ for all } \alpha \in A_\gamma\}$. For any $(Q_\gamma, A_\gamma) \in \mathcal{S}_\eta^{\ge 1}$, write $\tropstrat{\eta}{\gamma}$ for the relative interior of $\Phi (Q_\gamma, A_\gamma )$ so that
\begin{align}
\tropvar{\eta} = \bigcup_{\gamma \in \Gamma} \tropstrat{\eta}{\gamma}.
\end{align}
\begin{figure}[h]
	\includegraphics[scale=.8]{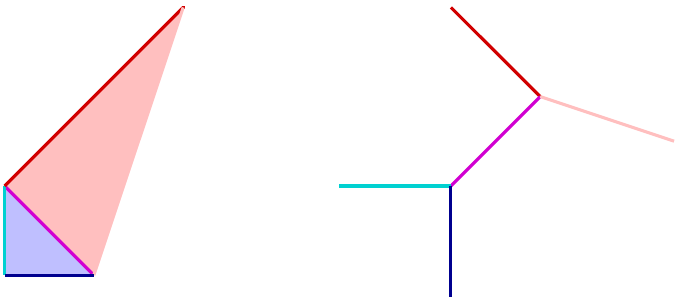}
	\caption{\label{fig:triangulation} A triangulation and tropical hypersurface induced by $\eta$}
\end{figure}
\begin{example} \label{example:defs}
	Consider an example of $A = \{(0,0), (1,0), (0, 1), (2, 3)\}$, and $Q$ its convex hull. Take $\eta : A \to \mathbb{R}$ equal to 0 except at $(2,3)$, where it equals 1. That gives the coherent triangulation and tropical hypersurface illustrated in Figure~\ref{fig:triangulation}. The correspondence $\Phi$ is indicated by the coloring of the simplices in the triangulation and the faces of the tropical hypersurface.
\end{example}

Turning to complex polynomials, given a Laurent polynomial $f = \sum_{m \in \mathcal{M}} c_m z^m$ where $\mathcal{M}$ is a finite set in $M$ and $c_m \neq 0$ for all $m \in \mathcal{M}$,  we say that $(\chull (\mathcal{M}), \mathcal{M})$ is the \textbf{marked Newton polytope} of $f$. Given any $A \subseteq \mathcal{M}$, the restriction of $f$ to $A$ is the polynomial 
\begin{align}
f_{A} = \sum_{a \in A} c_a z^a. 
\end{align}

\begin{figure}[h]
	\includegraphics[scale=.3]{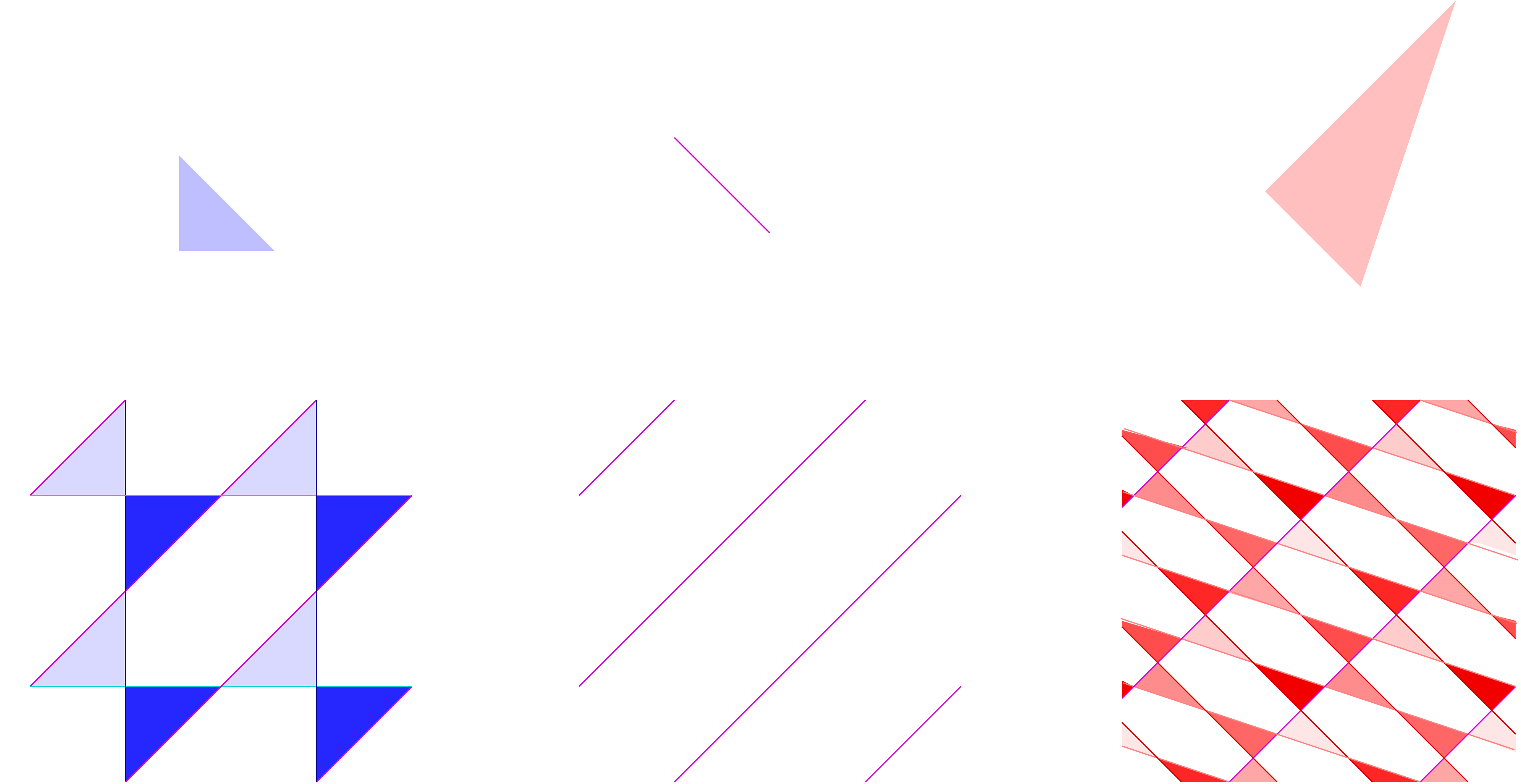}
	\caption{\label{fig:coamoebaG}Coamoebas associated to simplices.}
\end{figure}
Given a Laurent polynomial $f$, its zero locus $\comhyp{f}$ lives in the complex torus $N_{\mathbb{C}^*}$. In the notation from \ref{sec:coamoeba} we have the argument map $\Arg : N_{\mathbb{C}^*} \to N_{\mathbb{T}}$ to the real $n$-torus $N_\mathbb{T}$. For any polynomial $f$, we define its \textbf{coamoeba} $\coam{f} \subset N_\mathbb{T}$ to be the closure of the image of $\comhyp{f}$ under the argument map. 

\begin{example} \label{example:coams}
	As will be shown in the next subsection, the coamoeba of a simplex is a finite cover of the product of the coamoeba of a pair-of-pants and a torus. Take $f$ to be a generic polynomial with marked Newton polytope $(A,Q)$ from Example~\ref{example:defs}. Consider the restriction of $f$ to three simplices as indicated in Figure~\ref{fig:coamoebaG}. The corresponding coamoebas in the cover $N_\mathbb{R}$ of $N_\mathbb{T}$ are illustrated. 
\end{example}
\begin{definition} \label{def:phasetropical} 
	Let $f \in \mathbb{C}[M]$ and $(Q,A)$ be its marked Newton polytope. Given a coherent triangulation $\mathcal{S}_\eta = \{(Q_\gamma , A_\gamma ): \gamma \in \Gamma\}$ induced by $\eta : A \to \mathbb{R}$,  the \textbf{phase tropical hypersurface} of $f$ defined by $\eta$ is
	\begin{align*}
	\tphyp{\eta}{f} := \bigcup_{\gamma \in \Gamma} \tropstrat{\eta}{\gamma} \times \coam{f_{A_\gamma}} \subset N_\mathbb{R} \times N_\mathbb{T} = N_{\mathbb{C}^*} .
	\end{align*}
\end{definition}
\begin{example} \label{example:pthyp} Combining Examples \ref{example:defs} and \ref{example:coams} we obtain a partial picture of a phase tropical hypersurface $\tphyp{\eta}{f}$ illustrated in Figure~\ref{fig:ptvar}. Here we have not illustrated the coamoebas over the non-compact faces of the tropical hypersurface.
\end{example}
\begin{figure}[h]
	\includegraphics[scale=.3]{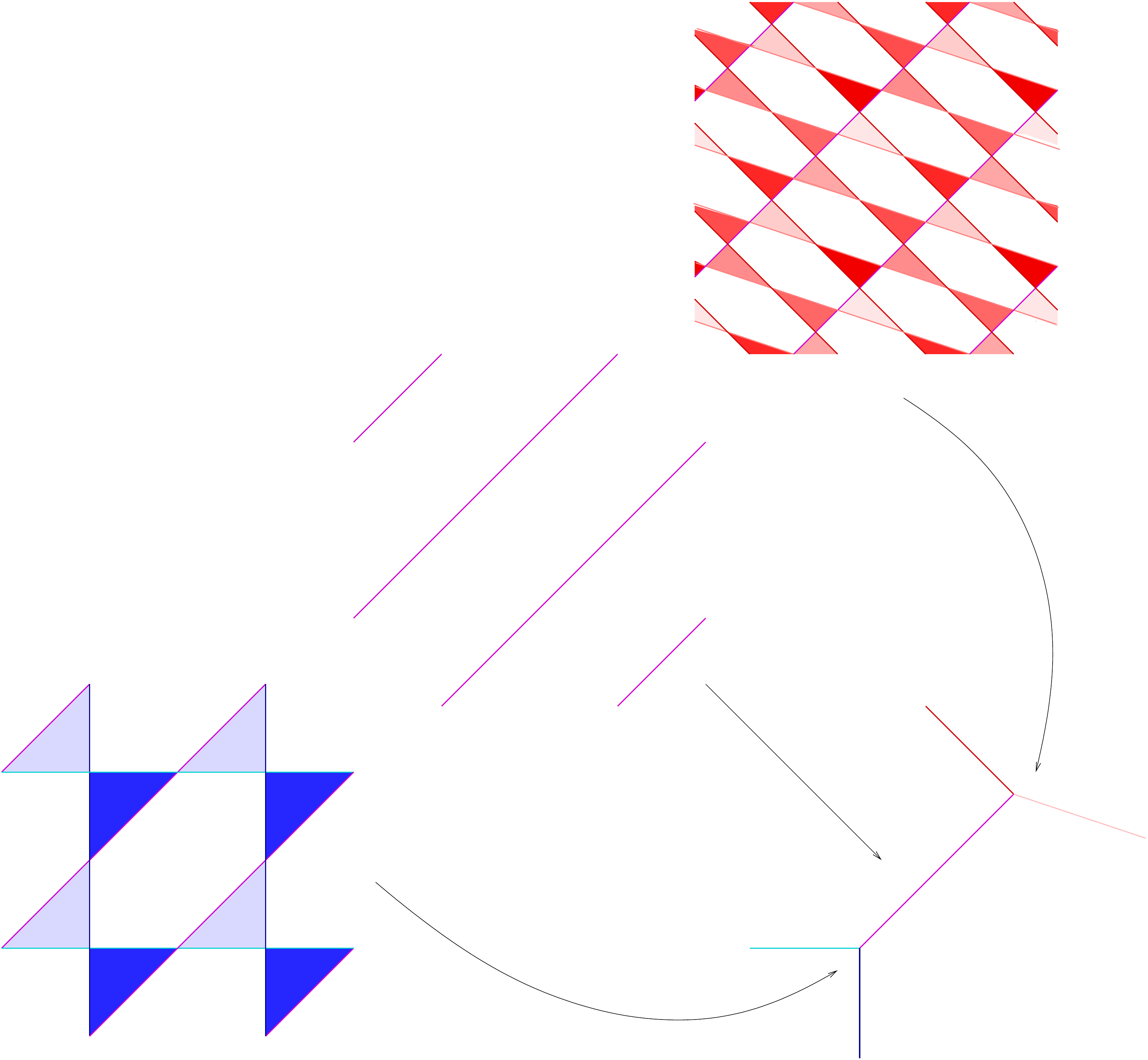}
	\caption{\label{fig:ptvar} The phase tropical hypersurface.}
\end{figure} 

We will also consider a compactified version of the phase tropical  hypersurface. For this, we identify $N_\mathbb{R} \times N_\mathbb{T}$ with $N_{\mathbb{C}^*}$ via the exponential map and consider the algebraic moment map 
$
\mu : N_{\mathbb{C}^*} \to M_\mathbb{R}
$
defined as 
\begin{align*}
\mu (z) = \frac{1}{\sum_{a \in A} | z^a |} \sum_{a \in A} |z^a | a.
\end{align*}
It is clear that $\mu$ is independent of the $N_\mathbb{T}$ factor and, for any $\theta \in N_\mathbb{T}$, we denote the restriction of $\mu$ to $N_\mathbb{R} \times \{\theta\}$ by $\mu_\mathbb{R} : N_\mathbb{R} \to M_\mathbb{R}$. It is also easy to observe that the image of $\mu$ is the relative interior of $Q$. If we wish to reference the marked polytope in the notation, we write $\mu^{A}$ for $\mu$ and $\mu_\mathbb{R}^A$ for $\mu_\mathbb{R}$.
\begin{definition} \label{def:compactphasetropical} Let $Q$ be an $n$-dimensional polytope. The \textbf{compactified phase tropical hypersurface} $\ctphyp{\eta}{f}$ is the closure
	\begin{align*}
	\overline{\left(\mu_\mathbb{R} \times Id \right)(\tphyp{\eta}{f} )} \subset Q \times N_\mathbb{T} 
	\end{align*}
	of the phase tropical  hypersurface in $Q \times N_\mathbb{T}$.
\end{definition}
Of course, we may also compactify the complex hypersurface $\comhyp{f} \subset N_{\mathbb{C}^*}$ by taking its closure under the image of the moment map. When the Newton polytope of $f$ is $n$-dimensional, we call 
\begin{align}
\ccomhyp{f} = \overline{\mu (\comhyp{f})} \subset Q \times N_\mathbb{T}
\end{align}
the compactification of $\comhyp{f}$. 

The boundaries of $\ccomhyp{f}$ and $\ctphyp{\eta}{f}$ both have stratifications indexed by the face lattice of $Q$. It will be helpful later on to describe these boundary strata intrinsically. Suppose $(Q,A)$ is a marked polytope in $M$, not necessarily full dimensional, $f$ a polynomial whose marked Newton polytope contains $(Q,A)$ and $\eta$ is any function on $A$.  Let 
$$N(A) := \{\mathbf{x} \in N  :  a (\mathbf{x}) = a^\prime (\mathbf{x}) \text{ for all }a, a^\prime \in A\}
$$ 
be the sublattice orthogonal to the affine span of $A$, and $N_\mathbb{K} (A) = N (A) \otimes_\mathbb{Z} \mathbb{K}$. Then it is easy to see that the tropical hypersurface $\tropvar{\eta} \subset N_\mathbb{R}$ is invariant under translations by $N_\mathbb{R} (A)$. Define the space
\begin{align}
\tropstrats{\eta}{f}{A} = \left\{ (\mathbf{x} + N_\mathbb{R} (A), \mathbf{\theta}) \in \frac{N_\mathbb{R}}{N_\mathbb{R} (A)}  \times N_\mathbb{T} : (\mathbf{x}, \mathbf{\theta}) \in \tphyp{\eta}{f_A} \right\}.
\end{align}
We may compactify $\tropstrats{\eta}{f}{A}$ by using the moment map $\mu^A$ associated to $A$. More explicitly, let $\iota_\mathbb{R} : N_\mathbb{R} / N_\mathbb{R} (A) \to N_\mathbb{R}$ be any section of the quotient map. Then define
\begin{align}
\tilde{\mu}^A_\mathbb{R} : \frac{N_{\mathbb{R}}}{N_{\mathbb{R}} (A) } \to Q
\end{align}
to be the composition $\mu^A_\mathbb{R} \circ \iota_\mathbb{R}$. As $\mu^A_\mathbb{R}$ is invariant with respect to translations by $N(A)$, it is clear that $\tilde{\mu}^A_\mathbb{R}$ is independent of the choice of $\iota$. Note that the image of $\tilde{\mu}^A_\mathbb{R}$ is the relative interior of $Q$. For the compactification of $\tropstrats{\eta}{f}{A}$ we take 
\begin{align}
\ctropstrats{\eta}{f}{A} = \overline{(\tilde{\mu}^A_\mathbb{R} \times Id) (\tropstrats{\eta}{f}{A})}.
\end{align}

One may relate $\ctropstrats{\eta}{f}{A}$ to the lower dimensional phase tropical hypersurface associated to $f$. To do this, let $M (A)$ be the saturation of the affine lattice spanned by $A$, $(\tilde{Q}, \tilde{A})$ the image of $A$ in $M (A)$ and $\tilde{\eta} : \tilde{A} \to \mathbb{R}$ the function equal to $\eta$, then $\ctropstrats{\eta}{f}{A}$ is homeomorphic to $\ctphyp{\tilde{\eta}}{f} \times N_\mathbb{T} (A)$. To define a homeomorphism, one can  use a lift $\iota_\mathbb{R} : N_\mathbb{R} / N_\mathbb{R} (A) \to N_\mathbb{R}$ to equate the argument factor of $\tropstrats{\eta}{f}{A}$ with $N_{\mathbb{T}} / N_\mathbb{T} (A) \times N_\mathbb{T} (A)$.
	
If $(Q^\prime, A^\prime)$ is a face of $(Q, A)$ where the marked Newton polytope of $f$ contains $(Q,A)$, then there is a natural inclusion
\begin{align} \label{eq:ptstrat_maps}
i_{A^\prime, A} : \ctropstrats{\eta}{f}{A^\prime} \to \ctropstrats{\eta}{f}{A}
\end{align}
whose image is the inverse image of $Q^\prime$ in $Q \times N_\mathbb{T}$ under the moment map. To define this map, it suffices to consider the non-compact strata. But there is already a map $ (\tilde{\mu}_\mathbb{R}^{A^\prime} \times Id) : \tropstrats{\eta}{f}{A^\prime} \to Q^\prime \times N_\mathbb{T} \subset Q \times N_\mathbb{T}$ and this maps bijectively onto  $\ctropstrats{\eta}{f}{A}$ over the relative interior of $Q^\prime$. 

Thus we obtain a stratification of the compactified phase tropical hypersurface 
\begin{align} \label{eq:pt_strat}
\ctphyp{\eta}{f} = \bigcup_{(Q^\prime, A^\prime) \text{ a face of } (Q, A)} \tropstrats{\eta}{f}{A^\prime}
\end{align}
whose strata are in bijective correspondence with the positive dimensional faces of $Q$. 

Turning to the geometry of complex hypersurfaces, we note that their tropical compactifications also carry a stratification by the boundary faces of $Q$. Here, assume $f$ is a Laurent polynomial with marked Newton polytope containing the face $(Q,A )$. We denote by $\comstrat{f}{A}$ the quotient of the hypersurface $\comhyp{f_{A}}$ in $N_{\mathbb{C}^*} = N_\mathbb{R} \times N_\mathbb{T}$ by $N_\mathbb{R} (A)$ (note that $f_{A}$ is homogeneous with respect to this action, so that the quotient is well defined). The space $\comstrat{f}{A}$ is not compact, but again we may compactify by taking its closure under the restricted moment map
\begin{align}
\ccomstrat{f}{A} = \overline{(\tilde{\mu}^A_\mathbb{R} \times Id) (\comstrat{f}{A})}.
\end{align}

As in the phase tropical setting, for a face $(Q^\prime, A^\prime)$ of $(Q, A)$, and marked Newton polytope of $f$ containing $(Q,A)$, there are natural maps 
\begin{align} \label{eq:comstrat_maps}
j_{A^\prime, A} : \ccomstrat{f}{A^\prime}\to \ccomstrat{f}{A}
\end{align}
defined analogously to those in \eqref{eq:ptstrat_maps}. The associated stratification is then
\begin{align} \label{eq:com_strat}
\ccomhyp{f} = \bigcup_{(Q^\prime, A^\prime) \text{ a face of } (Q, A)} \comstrat{f}{A^\prime}.
\end{align}

Even for smooth hypersurfaces $\comhyp{f}$, the tropical compactification may contain unwanted singularities on the boundary strata. To prevent such singularities, we call a Laurent polynomial $f$ \textbf{non-degenerate} if $f_{A'}$ is regular at $0$ for all faces $(Q^\prime, A^\prime)$ of $(Q,A)$.

\begin{theorem} \label{thm:maintheorem} Given a non-degenerate polynomial $f$ with marked Newton polytope $(Q,A)$ and $\eta : A \to \mathbb{R}$ defining a coherent triangulation, there are homeomorphisms $\psi$, $\bar{\psi}$ for which the diagram
	\begin{equation}
	\label{diag:commtropcomp2} 
	\begin{tikzpicture}[baseline=(current  bounding  box.center), node distance=2cm, auto]
	\node (A) {$\comhyp{f}$};
	\node (B) [right of=A]{$\tphyp{\eta}{f}$}; 
	\node (C) [below of=A]{$\ccomhyp{f}$};
	\node (D) [right of=C]{$\ctphyp{\eta}{f}$}; 
	\path[->,font=\scriptsize]
	(A) edge node[above] {$\psi$} (B)
	(C) edge node[below] {$\bar{\psi}$} (D);
	\path[right hook->,font=\scriptsize]
	(A) edge node[above] {} (C)
	(B) edge node[below] {} (D);
	\end{tikzpicture}
	\end{equation}
	commutes.
\end{theorem}

\subsection{Simple hypersurfaces} \label{sec:simplehyp} Often when considering pair-of-pants decompositions induced by a coherent triangulation $\mathcal{S} = \{(Q_\gamma, A_\gamma): \gamma \in \Gamma\}$ of a marked polytope $(Q,A)$ (e.g. as in \cite{PP}), the simplices $(Q_\gamma, A_\gamma)$ are required to have volume $\frac{1}{n!}$ (or normalized volume $1$). Such triangulations are referred to as maximal or unimodular. In practice, unimodular triangulations are comparatively rare and there will usually be several simplices in any given triangulation with larger volume. Indeed, there are many cases of marked polytopes without a single unimodular triangulation. For a simplex of volume greater than $\frac{1}{n!}$, the associated phase tropical hypersurface is no longer a pair-of-pants, but rather a finite cover of the pair-of-pants, called a \text{simple hypersurface} in \cite{NS}. Following loc. cit., we take a moment to consider this cover and that of the associated hypersurface $\comhyp{f}$.

First, let us establish some notation. Let $(Q, B)$ be a marked simplex in $M \cong \mathbb{Z}^n$ for which $B$ affinely spans $M_\mathbb{R}$ and 
\begin{align*}
f = \sum_{b \in B} c_b z^b \in \mathbb{C}[M]
\end{align*}
with $c_b \ne 0$ for every $b \in B$.  Fix an ordering of $B = \{b_0, \ldots, b_n\}$, write $c_i$ for $c_{b_i}$ and identify any subset $I \subseteq \hat{n}$ with its corresponding subset $\{b_i: i \in I\} \subseteq B$. Consider the map $\phi_B : N_{\mathbb{C}^*} \to \left( \mathbb{C}^* \right)^{n + 1}/ \mathbb{C}^* \subset \mathbb{P}^n$ defined by
\begin{align}
\phi_B (z) := [c_0 z^{b_0}, \cdots , c_n z^{b_n}]. 
\end{align}
One notes that $\phi_B$ extends via the inverse of the moment map to $\bar{\phi}_B : Q \times N_\mathbb{T} \to \Delta \times \mathbb{T}^n$ where $\Delta$ is the standard simplex.

Write $\Xi_B \subseteq M$ for the sublattice which is the $\mathbb{Z}$-span of $\{b_i - b_j : b_i, b_j \in B \}$. Then there are containments $N \subseteq \Xi_B^\vee \subset N_\mathbb{R}$ and we write $\Lambda_B$ for the image of $\Xi_B^\vee$ in the quotient $N_\mathbb{T} = N_\mathbb{R} / N$. Then, using notation from Section~\ref{sec:coamoeba}, we have the following basic result.

\begin{lemma} \label{lemma:complexcover}
	The maps $\phi_B$ and $\bar{\phi}_B$ are quotient maps by $\Lambda_B$. Furthermore, for any subset $I \subseteq \{0,\dots, n\}$ with $|I| \geq 2$, the restriction of $\phi_B$ to  $\comhyp{f_I}$ is a covering map to the complex pair-of-pants $P^{n - 1}_I$.
\end{lemma}
\begin{proof} To verify the statement that $\phi_B$ is the quotient map, first observe that  $\phi_B (z) = [c_0 z^{b_0}, \dots , c_n  z^{b_n}] = [c_0 , c_1 z^{b_1 - b_0} , \dots ,åc_n z^{b_n - b_0}]$.   This implies that $\phi_B$ factors through the homomorphism $N_{\mathbb{C}^*} \to (\mathbb{C}^*)^n$ given by $z \mapsto (z^{b_1 - b_0}, \dots, z^{b_n - b_0})$ which has kernel $\Lambda_B$. The extension to $\bar{\phi}_B$ follows immediately since $\Lambda_B \subset N_\mathbb{T}$ acts only on the $N_\mathbb{T}$ factor. 
	
	The fact that $\phi_B$ and $\bar{\phi}_B$ restrict to give covering maps from the hypersurface $\comhyp{f_I}$ to the pair-of-pants follows from the definition of $P^{n - 1}_I \subset (\mathbb{C}^*)^{n + 1} / \mathbb{C}^* \subset \mathbb{P}^n$ as the zero locus of $\sum_{i \in I} z_i = 0$. In particular, $f_I = \phi_B^* ( \sum_{i \in I} z_i )$ so that  $\phi_B$ restricts to $\comhyp{f_I}$ to give a well defined and an onto map.
\end{proof}

We now turn our attention to the phase tropical hypersurface of a marked simplex $(Q,B)$ with the function $\eta : B \to \mathbb{R}$ and a polynomial $f$. We first make an observation that the tropical hypersurface $\tropvar{\eta}$ depends on $\eta$ only up to a translation in $N_\mathbb{R}$. In particular, since $\eta$ is defined on a simplex, it is the restriction of an affine function $n_\eta + c$ on $M_\mathbb{R}$ to $B$ where $n_\eta \in N_\mathbb{R}$ and $c$ is a constant. In this instance, one observes that $\tropvar{\eta} = \tropvar{\mathbf{0}} + n_\eta$. The coamoeba $\coam{f_{I}}$ is independent of the function $\eta$ and only depend on $I$ and the arguments of the coefficients $\{c_i : i \in I\}$ of $f$. Thus, there is an elementary homeomorphism  $\tphyp{\eta}{f} \cong \tphyp{\mathbf{0}}{f}$
given by translating by $n_\eta$ in the $N_\mathbb{R}$ factor of $N_\mathbb{R} \times N_\mathbb{T}$. Consequently, the topology of the phase tropical hypersurface of a simplex is independent of the function $\eta$. For convenience, we choose $\eta_f : B \to \mathbb{R}$ to be the function $\eta_f (b_i) = - \log |c_i|$. 

\begin{lemma}\label{lemma:phasetropicalcover}
	The restriction of $\bar{\phi}_B$ to $\ctphyp{\eta_f}{f}$ is a covering map onto $\tt \bar \pp^{n-1}$.
\end{lemma}

\begin{proof}
	View the map $\phi_B $ as a map from $N_\mathbb{R} \times N_\mathbb{T}$ to $(\mathbb{C}^*)^{n + 1} / \mathbb{C}^* \subset  \mathbb{P}^n$. Consider the linear map $\xi : \mathbb{R}^{n + 1} \to M_\mathbb{R}$ defined by $\xi (e_i) = b_i$ and let $\mathbf{c} = \sum_i \log |c_i| e_i \in \mathbb{R}^{n + 1}$. It is then clear that the diagram 
	\begin{equation}
	\label{diag:moments}
	\begin{tikzpicture}[baseline=(current  bounding  box.center), node distance=2cm, auto]
	\node (A) {$N_\mathbb{R} \times N_\mathbb{T}$};
	\node (B) [right of=A,xshift=1cm]{$(\mathbb{C}^*)^{n + 1} / \mathbb{C}^*$}; 
	\node (C) [below of=A]{$N_\mathbb{R}$};
	\node (D) [right of=C,xshift=1cm]{$\mathbb{R}^{n + 1} / \mathbb{R}$}; 
	\path[->,font=\scriptsize]
	(A) edge node[above] {$\phi_B$} (B)
	(C) edge node[below] {$\xi^\vee + \mathbf{c}$} (D)
	(A) edge node[left] {$\pi_1$} (C)
	(B) edge node[right] {$\Log$} (D);
	\end{tikzpicture}
	\end{equation}
	commutes where $\pi_1$ is projection to the first factor. Furthermore, the composition of the affine map $\xi^\vee + \mathbf{c}: N_\mathbb{R} \to \mathbb{R}^{n + 1}$ with every dual basis vector $e_i^\vee : \mathbb{R}^{n + 1} \to \mathbb{R}$ yields the map $b_i - \eta_f (b_i)$. Thus, the tropical polynomial for the standard pair-of-pants pulls back to $F_{\eta_f}$ and $\xi^\vee + \mathbf{c}$ maps the tropical hypersurface $\tropvar{f}$ to the tropical pair-of-pants $\pp^{n - 1}$. 
	
	For $I \subseteq \{0, \ldots, n\}$ with $|I| \geq 2$, denote by $Q_I$ the convex hull of the corresponding set in $B$. Then, over the face $\Phi (Q_I, I)$ of the tropical hypersurface, in the phase tropical hypersurface $\tphyp{f}{\eta_f}$, lies the coamoeba $\coam{f_{I}}$. Utilizing Lemma~\ref{lemma:complexcover} and a commutative diagram analogous to \eqref{diag:moments} with argument maps, we have that $\coam{f_{I}}$ maps to the coamoeba $\cc_I$. This implies the result.
\end{proof}
Combining Lemmas~\ref{lemma:complexcover} and \ref{lemma:phasetropicalcover}, we obtain an extension of Theorem~\ref{thm:homeo} to the simple hypersurface case.
\begin{theorem} \label{theorem:homeo_simple_hypersurfaces}
	Given a marked simplex $(Q, B)$ which affinely spans $M_\mathbb{R}$ and any $\eta : B \to \mathbb{R}$, there is a homeomorphism $\bar{\psi} :\ccomhyp{f} \to \ctphyp{\eta}{f}$.
\end{theorem}
\begin{proof}
	It suffices to prove this theorem in the non-compact case.
We write $\phi : P^{n - 1} \to \tt \pp^{n -1}$ for the homeomorphism in Theorem~\ref{thm:homeo}. Note that both inclusions $P^{n -1} \hookrightarrow \left( \mathbb{C}^* \right)^{n + 1}/ \mathbb{C}^*$ and $\tt \pp^{n -1} \hookrightarrow \left( \mathbb{C}^* \right)^{n + 1}/ \mathbb{C}^*$ induce isomorphisms on first homology (and, on the fundamental group when $n > 2$). This follows by looking at the coamoeba (see, e.g., \cite{Sheridan})
for the complex and phase tropical pair-of-pants and its covering in $\mathbb{R}^{n + 1}$, which is simply connected for $n \geq 3$ and is the universal abelian cover for $n = 2$. Moreover, it is evident from the construction of $\phi$ that the diagram 
	\begin{equation}
	\label{diag:homology_commutative} 
	\begin{tikzpicture}[baseline=(current  bounding  box.center), node distance=2cm, auto]
	\node (A) {$H_1 (P^{n -1} ; \mathbb{Z} )$};
	\node (B) [right of=A,xshift=2cm]{$H_1 (\tt \pp^{n -1}; \mathbb{Z})$}; 
	\node (C) [below of=A, right of=A]{$H_1 (\left( \mathbb{C}^* \right)^{n + 1}/ \mathbb{C}^* ; \mathbb{Z}) \cong \mathbb{Z}^n$};
	\path[->,font=\scriptsize]
	(A) edge node[above] {$H(\phi)$} (B)
	(A) edge node[below, xshift=-.2cm] {$\cong$} (C)
	(B) edge node[below, xshift=.2cm] {$\cong$} (C);
	\end{tikzpicture}
	\end{equation}
	commutes.
	
	By Lemmas~\ref{lemma:complexcover} and \ref{lemma:phasetropicalcover}, $\comhyp{f}$ and $\tphyp{\eta}{f}$ are covers of $P^{n - 1}$ and $\tt \pp^{n -1}$ obtained by pulling back the subspaces along the cover $\phi_B : N_{\mathbb{C}^*} \to \left( \mathbb{C}^* \right)^{n + 1}/ \mathbb{C}^*$. This cover corresponds to a sublattice of $H_1 (\left( \mathbb{C}^* \right)^{n + 1}/ \mathbb{C}^* ; \mathbb{Z}) \cong \pi_1 (\left( \mathbb{C}^* \right)^{n + 1}/ \mathbb{C}^*)$. The commutativity of \eqref{diag:homology_commutative}, pulled back along the Hurewicz homomorphism, then implies that $\pi_1 (\phi)$ takes the normal subgroup associated to $\phi_B|_{\comhyp{f}}$ to that of $\phi_B|_{\tphyp{\eta}{f}}$ implying the result.
\end{proof}
The arguments given in Lemmas~\ref{lemma:complexcover} and \ref{lemma:phasetropicalcover} easily extend to the strata $\ccomstrat{f}{B^\prime}$ and $\ctropstrats{\eta}{f}{B^\prime}$. We record the stratified version of Theorem~\ref{theorem:homeo_simple_hypersurfaces} as a corollary.

\begin{corollary} \label{cor:strat_homeo} Given a marked simplex $(Q, B)$ contained in the marked Newton polytope of $f$, $\eta : B \to \mathbb{R}$ and any sub-simplex $(Q^\prime, B^\prime)$, there are homeomorphisms $\bar{\psi}_{B^\prime}$ and $\bar{\psi}_B$ for which
	\begin{equation}
	\label{diag:strat_commutative} 
	\begin{tikzpicture}[baseline=(current  bounding  box.center), node distance=2cm, auto]
	\node (A) {$\ccomstrat{f}{B^\prime}$};
	\node (B) [right of=A,xshift=1cm]{$\ctropstrats{\eta}{f}{B^\prime}$}; 
	\node (C) [below of=A]{$\ccomstrat{f}{B}$};
	\node (D) [right of=C,xshift=1cm]{$\ctropstrats{\eta}{f}{B}$}; 
	\path[->,font=\scriptsize]
	(A) edge node[above] {$\bar{\psi}_{B^\prime}$} (B)
	(C) edge node[below, xshift=-.2cm] {$\bar{\psi}_B$} (D);
	\path[right hook->,font=\scriptsize]
	(A) edge node[left] {$j_{B^\prime, B}$} (C)
	(B) edge node[right] {$i_{B^\prime, B}$} (D);
	\end{tikzpicture}
	\end{equation}
	commutes.
\end{corollary} 

\subsection{Proof of Theorem~\ref{thm:maintheorem}} \label{sec:proof}
Having extended Theorem~\ref{thm:homeo} to simple hypersurfaces and their stratified compactifications, we now apply the results of Milkalkin \cite{PP} based on Viro's patchworking method \cite{Viro83}  to obtain Theorem~\ref{thm:maintheorem} for general hypersurfaces.
\begin{proof}[Proof of Theorem~\ref{thm:maintheorem}]
	One first observes that, since $f$ is assumed to be non-degenerate, $H_f$ is diffeomorphic to $H_g$ for any other non-degenerate polynomial $g$ with marked Newton polytope $(Q,A)$. To see this, one can compactify $N_{\mathbb{C}^*}$ to the toric variety $Y_Q$ and resolve all singularities of the toric boundary to obtain $X_Q$ with a normal crossing divisor $D$. Then Laurent polynomials with Newton polytope $Q$ may be identified with a dense open subset of  sections of the line bundle $\mathcal{O} (1)$ determined by $Q$. Moreover, the condition of non-degeneracy implies that the zero locus $Z_f$, which is an analytic compactification of $H_f$, transversely intersects the divisor $D$. Taking a family $\xi : [0, 1] \to \Gamma (X_Q, \mathcal{O} (1))$ of such non-degenerate sections, one may consider the incidence variety $\mathcal{Y} = \{(t, z ) : z \in Z_{\xi (t)}\}$ along with the function $\pi : \mathcal{Y} \to [0,1]$ induced by projection. Let $\mathcal{D} = \{(t, z) \in \mathcal{Y}: z \in D\}$. By the openness of the transversality condition,  $\pi$ and $\pi |_{D}$ are trivial families and, equipping $\mathcal{Y}$ with a connection for which $\mathcal{D}$ is horizontal and taking parallel transport gives a diffeomorphism of the pair $(Z_{\xi (0)}, Z_{\xi (0)} \cap D)$ with $(Z_{\xi (1)}, Z_{\xi (1)} \cap D)$. Excising the respective subspaces then produces the diffeomorphism.
	
	Next we note that, for $t \in \R_{>0}$ and small, the polynomial $f_t = \sum_{a \in A} c_a t^{\eta (a)} z^a$ is non-degenerate, regardless of the coefficients $\mathbf{c}_\mathbb{C} := (c_a) \in \mathbb{C}^A$ (cf. \cite{GKZ}). In particular, an alternative definition of a non-degenerate polynomial $f$ is that the principal $A$-determinant $E_A (f)$ is non-zero. By \cite[Theorem~10.1.4]{GKZ}, $E_A$ is a polynomial in the coefficients $(c_a)$ whose Newton polytope is the secondary polytope $\Sigma (A)$. It can be shown that $\Log$ of the coefficients $\mathbf{c}_\mathbb{C} t^{\eta}$ of $f_t$ lie in the interior of the cone dual to the triangulation defined by $\eta$ for sufficiently small $t$. This implies that, for such $f_t$, the $\Log (\mathbf{c}_\mathbb{C})$ lies outside the amoeba of $E_A$ implying $E_A (\mathbf{c}_\mathbb{C}) \ne 0$.
	
	To complete the proof we apply the reconstruction result \cite[Theorem~4]{PP}, adapted to the non-unimodular case. In this modified form, it asserts that for $\eta$ inducing the coherent triangulation $\mathcal{S} = \{(Q_\gamma, A_\gamma) : \gamma \in \Gamma\}$, there is homeomorphism between $\ccomhyp{f}$ and the topological direct limit (i.e. the colimit in the category of topological spaces). Thus we obtain 
	\begin{align} \label{eq:complex_limit}
	\ccomhyp{f} \approx \lim_{\to} \ccomstrat{f}{A_\gamma}.
	\end{align}
	Here we mean that one may consider the face lattice $\Gamma$ of $\mathcal{S}$ given by inclusions as a category and $\ccomstrat{f}{-}$ as a functor from $\Gamma$ to topological spaces. Then the limit of this functor is achieved by gluing simple hypersurfaces along common boundary strata.

	Using the dual subdivision of $N_\mathbb{R}$ to that given by the tropical hypersurface, one obtains a decomposition $\bar{\mathcal{H}}_\eta = \cup_{\gamma \in \Gamma} Y_\gamma$. We then have that $\ctphyp{\eta}{f} = \cup_{\gamma \in \Gamma} \mathcal{T}  Y_\gamma$, where $\mathcal{T} Y_\gamma$ is the part of the phase tropical hypersurface lying over $Y_\gamma$.  Each $\mathcal{T} Y_\gamma$ can be identified with a partially contracted $\ctropstrats{\eta}{f}{A_\gamma}$ which is clearly homeomorphic to the original  $\ctropstrats{\eta}{f}{A_\gamma}$. These identifications are compatible with inclusion maps $i_{A_\gamma, A_{\tilde{\gamma}}}$, when $(Q_\gamma, A_\gamma)$ if a face $(Q_{\tilde{\gamma}}, A_{\tilde{\gamma}})$. In particular, we have that 
	\begin{align} \label{eq:phase_limit}
	\ctphyp{\eta}{f} \approx \lim_{\to} \ctropstrats{\eta}{f}{A_\gamma}.
	\end{align}
	Here again we regard this as a limit of the functor from the poset category $\Gamma$ to topological spaces taking $\gamma$ to $\ctropstrats{\eta}{f}{A_\gamma}$ and inclusions of faces $(Q_\gamma, A_\gamma)$ of $(Q_{\tilde{\gamma}}, A_{\tilde{\gamma}})$ to $i_{A_\gamma, A_{\tilde{\gamma}}}$. 
	
	By Corollary~\ref{cor:strat_homeo}, we have a natural isomorphism of functors from $\ccomstrat{f}{-}$ to $\ctropstrats{\eta}{f}{-}$ implying their limits are homeomorphic. Equations~\eqref{eq:complex_limit} and \eqref{eq:phase_limit} then give that $\ccomhyp{f} \approx \ctphyp{\eta}{f}$. Removing the boundary strata on both sides of the equation gives the open case.
\end{proof}

\subsection{Monodromy} As Example~\ref{example:pthyp} illustrates, the polyhedral complex of a phase tropical hypersurface controls much of the geometry and topology of the complex hypersurface. As another example of this, we briefly sketch a potential application of Theorem~\ref{thm:maintheorem}. The function $\eta$ on $A$ arises in toric geometry in order to produce a degeneration $F_\eta : \mathcal{Y} \to \mathbb{C}$ of a family of hypersurfaces with  $F_\eta^{-1} (1) = H_f$ (see, e.g. \cite{mumford}). In this setting,  an important invariant is the holonomy  $\xi_\eta : H_f \to H_f$ associated to $F_\eta$ and a connection on $\mathcal{Y}$.  A geometric characterization of the isotopy class of $\xi_\eta$, or the monodromy of the degeneration, has been the subject of several investigations in tropical geometry (cf. \cite{MZh}, \cite{japan}). The class has been identified, but here we produce an explicit map on phase tropical hypersurface. 

We note that every phase tropical hypersurface $\tphyp{\eta}{f}$ carries a distinguished automorphism $\mathcal{T} \xi_\eta : \tphyp{\eta}{f} \to \tphyp{\eta}{f}$ defined as 
\begin{equation}
\mathcal{T} \xi_\eta (n_1, n_2) = (n_1, n_2 + n_1).
\end{equation} 
With a parametrized version of Theorem~\ref{thm:maintheorem} in hand, one may establish the following:
\begin{conjecture}
	The following diagram 
	\begin{equation}
	\label{diag:monodromy} 
	\begin{tikzpicture}[baseline=(current  bounding  box.center), node distance=2cm, auto]
	\node (A) {$\comhyp{f}$};
	\node (B) [right of=A]{$\tphyp{\eta}{f}$}; 
	\node (C) [below of=A]{$\comhyp{f}$};
	\node (D) [right of=C]{$\tphyp{\eta}{f}$}; 
	\path[->,font=\scriptsize]
	(A) edge node[above] {$\psi$} (B)
	(C) edge node[below] {$\psi$} (D)
	(A) edge node[left] {$\xi_\eta$} (C)
	(B) edge node[right] {$\mathcal{T} \xi_\eta$} (D);
	\end{tikzpicture}
	\end{equation}
commutes up to isotopy.
\end{conjecture}

\end{document}